\documentclass[11pt]{amsart}
\usepackage{amsmath,latexsym,amsfonts,amssymb,mathrsfs}
\usepackage{graphicx,verbatim}
\usepackage{a4wide}
\usepackage[active]{srcltx}
\usepackage[colorlinks,linkcolor={blue},
citecolor={blue},urlcolor={red},]{hyperref}
\usepackage{hyperref}
\usepackage{fouridx}

\theoremstyle{plain}
\newtheorem{theorem}{Theorem}[section]
\theoremstyle{remark}

\theoremstyle{plain}
\newtheorem{corollary}[theorem]{Corollary}
\newtheorem{lemma}[theorem]{Lemma}
\newtheorem{proposition}[theorem]{Proposition}
\newtheorem{definition}[theorem]{Definition}

\newtheorem{assumption}{Assumption}
\numberwithin{equation}{section}

\newcommand{\HL}{\fourIdx{}{}{\!\!\!\mathrm{H}}{}\DDelta} 
\newcommand{\BL}{\fourIdx{}{}{\!\!\!\mathrm{B}}{}\DDelta}  
\newcommand{\LB}{\fourIdx{}{}{\!\mathrm{LB}}{}\Delta}  
\newcommand{\Ric}{\mathrm{Ric}}

\newcommand\dela[1]{}
\newcommand\coma[1]{\color{red} {#1 }}
\newcommand\deld[1]{}

\newcommand{\be}{\begin{equation}}
\newcommand{\ee}{\end{equation}}
\newcommand{\ba}{\begin{array}}
\newcommand{\ea}{\end{array}}
\newcommand{\beas}{\begin{eqnarray*}}
\newcommand{\eeas}{\end{eqnarray*}}
\newcommand{\bea}{\begin{eqnarray}}
\newcommand{\eea}{\end{eqnarray}}

\newcommand{\lb}{\label}

\newcommand{\Curl}{{\mathrm{Curl}} \, }
\newcommand{\Curln}{{\mathrm{curl} \,}  }
\newcommand{\Div}{\mathrm{div} \, }
\newcommand{\Grad}{{\nabla} \, }
\newcommand{\DDelta}{\boldsymbol{\mathrm{\Delta}} \, }
\newcommand{\A}{\boldsymbol{\mathrm{A}}  }
\newcommand{\B}{\boldsymbol{\mathrm{B}}  }

\newcommand{\CC}{\boldsymbol{\mathrm{C}}  }
\newcommand{\bL}{\boldsymbol{\mathrm{L}}  }
\newcommand{\loc}{\mathrm{loc}}

\newcommand{\Z}{{\mathbf Z}}
\newcommand{\N}{{\mathbf N}}
\newcommand{\nnabla}{\boldsymbol{\mathrm{\nabla}} \, }
\newcommand{\inprod}[2]{\left\langle{#1},{#2}\right\rangle}

\newcommand{\ve}{{\bf e}}
\newcommand{\vf}{{\bf f}}
\newcommand{\vg}{{\bf g}}

\newcommand{\vu}{{\bf u}}

\newcommand{\vv}{{\bf v}}
\newcommand{\vw}{{\bf w}}
\newcommand{\vy}{{\bf y}}
\newcommand{\vY}{{\bf Y}}

\newcommand{\ovz}{{\bf z}}
\newcommand{\x}{{\bf x}}
\newcommand{\xb}{{\bf x}}
\newcommand{\vphi}{{\bf \phi}}
\newcommand{\bphi}{\mbox{\boldmath $\phi$}}
\newcommand{\bpsi}{\mbox{\boldmath $\psi$}}
\newcommand{\btheta}{\mbox{\boldmath $\theta$}}
\newcommand{\xh}{\widehat{\x}}
\newcommand{\vomega}{\boldsymbol{\mathrm{\omega}} \, }
\newcommand{\bbP}{\mathbb{P}}
\newcommand{\calA}{\mathcal A}

\newcommand{\calD}{\mathcal D}
\newcommand{\calF}{\mathcal F}
\newcommand{\calG}{\mathcal G}
\newcommand{\calH}{\mathcal H}
\newcommand{\calV}{\mathcal V}
\newcommand{\calL}{\mathcal L}

\newcommand{\frakT}{\mathfrak T}
\newcommand{\tvf}{{\tilde{\vf}}}
\newcommand{\tvu}{{\tilde{\vu}}}

\newcommand{\tvv}{{\tilde{\vv}}}
\newcommand{\tvw}{{\tilde{\vw}}}
\newcommand{\hH}{\mathbb{H}}
\newcommand{\nN}{\mathbb{N}}
\newcommand{\E}{\mathbb{E}}
\newcommand{\R}{\mathbb{R}}
\newcommand{\lL}{\mathbb{L}}
\newcommand{\zZ}{\mathbb{Z}}
\newcommand{\what}{\widehat}
\newcommand{\nk}{ {n^{(k)}} }
\def\S{\mathbb S^2}
\def\la{\left(}
\def\ra{\right)}
\title[Stochastic Navier-Stokes on sphere]
{Random dynamical systems generated by stochastic Navier--Stokes equation 
on the rotating sphere}
\author{Z.~Brze\'{z}niak}
\address{Department of Mathematics \\
The University of York \\
Heslington, York, Y010 5DD \\
United Kingdom}
\email{zdzislaw.brzezniak@york.ac.uk}

\author{B.~Goldys}
\address{
School of Mathematics and Statistics \\
The University of Sydney, NSW 2006 \\
Australia}
\email{beniamin.goldys@sydney.edu.au}

\author{Q. T. Le Gia}
\address{%
School of Mathematics and Statistics \\
University of New South Wales \\
Sydney, NSW 2052, Australia}

\email{qlegia@unsw.edu.au}
\thanks{This work was partially supported by the ARC grant DP120101886 and DP120101816.\\
This work commenced during the first named author's visit to  the School of Mathematics and Statistics, 
University of New South Wales, Sydney and he  would like to thank 
 for the hospitality.}

\date{\today}
\begin{document}
\begin{abstract}
In this paper we first prove the existence and uniqueness of solution
to the stochastic Navier--Stokes equation on the rotating 2-dimensional sphere.
Then we show the existence of an asymptotically compact random dynamical system
associated with the equation.
\end{abstract}
\maketitle
\tableofcontents

\section{Introduction}

The aim of this work is to initiate systematic analysis of rotating stochastic
fluids on surfaces, notably on a sphere, and in thin shells. Importance of such problems for geophysical fluid dynamics and climate modelling are well known. This paper is concerned with two basic questions concerning the stochastic Navier--Stokes equation (SNS) on the 2-dimensional rotating sphere. The first one is about the existence and uniqueness of appropriately defined solutions and the second one is about the existence of stochastic flow and its asymptotic compactness.
\par\noindent
The deterministic Navier--Stokes
equations (NSEs) on the sphere has been object of intense study since early 1990's. Il'in and Filatov \cite{IliFil89,Ili91} considered the existence and uniqueness of solutions to these equations and estimated the Hausdorff dimension of their global attractors \cite{Ili94}. Temam and Wang \cite{TemWan93} considered the inertial forms of NSEs on sphere while Temam and Ziane \cite{TemZia97}, see also \cite{Avez+Bamberger_1978},  proved that the NSE on a 2-dimensional sphere is a limit
of NSEs defined on spherical shells \cite{TemZia97}. In other directions, Cao, Rammaha and Titi \cite{CaoRamTit99} proved the Gevrey regularity of the solution and found an upper bound on the asymptotic degrees of freedom for the long-time dynamics.

Concerning the numerical simulation of the deterministic NSEs on sphere, Fengler and Freeden \cite{FenFre05} obtained some impressive numerical results using the spectral method, while the numerical analysis of a pseudo-spectral method for these equations has been carried out in Ganesh, Le Gia and Sloan
in \cite{GanLeGSlo11}.

The question of the existence and uniqueness of solutions to the stochastic Navier--Stokes equation on 2D bounded
domains has been thoroughly investigated by many authors, beginning with the paper
by Bensoussan and Temam in 1973\cite{BenTem73}. In the case when continuous dependence on initial data remains open (for example, when the initial data is
merely in $L_2$), the existence of martingale solutions have been considered by
Capinski and Gatarek \cite{CapGat94}, Flandoli and Gatarek \cite{Fla94,FlaGat95}
and Mikulevicius and Rozovskii \cite{MikRoz05}. The uniqueness of
the martingale solution for the SNS on flat 2D bounded domains has been proved
by Ferrario in \cite{Fer03}. When working in spaces where continuous
dependence on the initial data can be expected, the existence of solutions
can sometimes be established on a preordained probability space. Such
solutions are often referred to as ``strong'' (in probabilistic sense)
or ``pathwise'' solutions. The existence of pathwise solutions in
$L_\infty([0,T],L_2)$ has been established by Da Prato and Zabczyk
\cite{PraZab96} and later by others \cite{Breckner00,MenSri02}. The existence of local
solutions evolving in Sobolev spaces, such as $W^{1,p}$, was addressed
by Mikulevicius and Rozovskii in \cite{MikRoz04} and
Brzezniak and Peszat in \cite{BrzPes00}.

The first result of our work is Theorem \ref{thm:existence} on the existence and uniqueness of pathwise variational solutions to the SNS equation on the rotating sphere for a general Gaussian noise. The only condition we require is, roughly speaking, that the corresponding linear stochastic equation (obtained by neglecting the nonlinear term in the SNS) has an $L^4$-solution. The fact that the equation is considered on the sphere and the presence of the Coriolis force lead to only minor modifications in the proof of existence. We can still use the classical Galerkin arguments based on expansions into series of vector spherical harmonics. In order to prove uniqueness we modify the classical argument of Lions and Prodi \cite{LioPro59}. Next, in Theorem \ref{thm:limit} we prove continuous dependence on the force and the driving noise.

The main result of this paper is presented in Theorem \ref{main}, where we prove that the random dynamical system associated to equation \eqref{equ:ode} is asymptotically
compact. This result is crucial for the proof the existence of a compact random attractor and the existence of an invariant measure. These questions are addressed in an accompanying paper \cite{arxiv}.  A similar result was obtained in \cite{robinson} for the deterministic NSE for a time-dependent force. A similar results for the stochastic NSE in unbounded domain was obtained in \cite{BrzLi06}

We emphasize, that the aforementioned results are obtained under the minimal assumptions on the regularity of the noise and of the external force. In particular, the external force $f$ is an element of of the dual $V^\prime$ of an appropriately defined energy space $V$.




\section{The Navier--Stokes equations on a rotating unit sphere}
\subsection{Preliminaries}

Let $\S$ be a $2$-dimensional unit sphere in $\R^3$,
i.e. $\S=\{ \xb=(x_1,x_2,x_3) \in \R^3: \vert \xb \vert=1\}$.
An arbitrary point $\xb$ on $\S$
can be parametrized by the spherical coordinates
\[
  \xb =\xh(\theta,\phi) =(\sin\theta\cos\phi,\sin\theta\sin\phi,\cos\theta),
  \quad 0\le \theta \le \pi,\; 0\le \phi \leq 2\pi.
\]
If $\xb =\xh(\theta,\phi)$ as above, then the corresponding angles $\theta$ and $\phi$ will be denoted by $\theta(\xb)$ and $\phi(\xb)$, or simply by $\theta$ and $\phi$. By $\xh$ we will also denote the unit normal vector field on $\S$, i.e. $\xh(\xb)=\xb$, for $\xb\in \S$.

Let $\btheta=\btheta(\theta,\phi)$ and $\bphi=\bphi(\theta,\phi)$
be the standard unit tangent vectors to $\S$ at point $\xh(\theta,\phi)\in \S$
 in the spherical
coordinates, that is
\[
\btheta = (\cos\theta\cos\phi,\cos\theta\sin\phi,-\sin\theta), \quad
\bphi = (-\sin\phi,\cos\phi,0).
\]

For a scalar function $f$ on $\S$, its surface gradient is given by
\[
 \Grad f =  \frac{\partial f}{\partial \theta} \btheta +
           \frac{1}{\sin \theta} \frac{\partial f}{\partial \phi} \bphi.
\]
Unless specified otherwise, by a vector field on $\S$ we mean a tangential vector filed, i.e. in the language of differential geometry, a section of the tangent vector bundle.

For a
vector field $\vu = (u_\theta,u_\phi)$ on $\S$, i.e.
$\vu = u_\theta \btheta + u_\phi \bphi$, its surface divergence
is given by
\[
 \Div \vu = \frac{1}{\sin\theta} \left( \frac{\partial}{\partial \theta} (u_\theta \sin \theta) +
                                        \frac{\partial}{ \partial \phi} u_\phi \right).
\]

The velocity field $\vu(\xh,t)=(u_\theta(\xh,t),u_\phi(\xh,t))$ of a geophysical
fluid flow on a rotating unit sphere $\S\subset\R^3$ under the external force
$\vf$ is governed by the Navier-Stokes equations (NSEs),
which is written as 
\cite{EbiMar70,Tay92}
\begin{equation}\label{NSE_sphere}
\partial_t \vu + \nabla_{\vu} \vu - \nu \bL \vu +
  \vomega \times \vu + \frac{1}{\rho} \Grad p = \vf, \quad  \Div \vu = 0,
  \quad \vu(\xh,0) = \vu_0.
\end{equation}
Here $\nu$ is the viscosity and $\rho$ is the density of the fluid, the normal
vector field $\vomega = 2\Omega (\cos\theta) \xb$ is the Coriolis acceleration,
$\Grad$ and $\Div$ are the
surface gradient and divergence, respectively.
The covariant derivative
$\nabla_{\vu} \vu$ is the nonlinear term in the equations. Here, the operator
$\bL$ is given by \cite{Tay92}
\be\label{equ:Lop}
\bL = \DDelta + 2 \Ric,
\ee
where $\DDelta$ is the Hodge Laplacian (or Laplace-de Rham) operator and $\Ric$ denotes the Ricci
tensor, which in case of the sphere coincides with the Riemannian tensor.
\par
It is easy to check that the Ricci tensor of the sphere is given by,
\be\label{equ:RicS}
 {\Ric}  = \left[\begin{array}{cc}
                            1  & 0 \\
                            0  & \sin^2\theta
                     \end{array} \right]
\ee

We consider now the nonlinear term $\nnabla_{\!\!\!\vu} \vu$. Given two vector fields $\vu$ and $\vv$ on $\S$,  by  \cite{DubNovFom86}, we can find  vector fields $\tilde{\vu}$ and $\tilde{\vv}$
defined in some neighbourhood of the surface $\S$ and such that their restrictions to $\S$ are equal to, respectively, $\vu$ and $\vv$, i.e.
 $\tilde{\vu}|_{\S} = \vu \in T\S$ and $\tilde{\vv}|_{\S} = \vv \in T\S$.
Then we put
\[
  [\nnabla_{\!\!\!\vv} \vu](\xb) = \pi_{\xb}
  \left( \sum_{i=1}^3 \tilde{\vv}_i(\xb) \partial_i \tvu(\xb) \right)
  = \pi_{\xb} \big( (\tvv(\xb) \cdot \tilde{\nabla})\tvu(\xb) \big),\;\; \xb \in \S,
\]
where $\tilde{\nabla}$ is the usual gradient in $\R^3$ and,  for $\xb \in \S$,
$\pi_{\xb}$ is the orthogonal projection from $\R^3$
onto the tangent space $T_{\xb} \S$ to $\S$ at $\xb$.

Note that, by decomposing $\tvu$ and $\tvv$ into tangential and normal
components of $T_{\xb}\S$, using the orthogonal properties of $\pi_{\xb}$,
one can show that
\begin{equation}\label{equ:piuv}
\pi_{\xb} (\tvu \times \tvv) =
\vu \times ( (\xb \cdot \vv) \xb) +
(\xb \cdot \vu) \xb \times \vv  \mbox{ for any }  \;\; \xb\in \S.
\end{equation}
In the above formula, as in many to follow, unless there is a danger of ambiguity, we omit the symbol $\xb$. For example, the LHS of \eqref{equ:piuv} should be read as $\pi_{\xb} (\tvu(\xb) \times \tvv(\xb))$.

We set $\vv = \vu$ and use the formula
\[
 (\tvu \cdot \tilde{\nabla})\tvu =   
   \tilde{\nabla} \frac{|\tvu|^2}{2} - \tvu \times (\tilde{\nabla} \times \tvu)
\]
to obtain
\begin{equation}\label{short_nonlin0}
        \nnabla_{\!\!\!\vu} \vu = \Grad \frac{|\vu|^2}{2} -
        \pi_{\xh} \big(\tvu \times (\nabla \times \tvu)\big)
\end{equation}

Using \eqref{equ:piuv} for the vector fields $\tilde{\vu}$ and
$\tvv = \tilde{\nabla} \times \tvu$, we
have
\begin{equation}\label{equ:conc1}
\pi_{\xb}(\tvu \times (\tilde{\nabla} \times \tvu))
= \vu \times ( (\xb \cdot (\tilde{\nabla} \times \tvu)) \xb)
 + (\tvu \cdot \xb) \xb \times (\tilde{\nabla} \times \tvu),\;\; \xb\in \S.
\end{equation}

With a \dela{tangential} vector field $\vu$ on $\S$, since the normal component
is zero, \eqref{equ:conc1} is reduced to
\begin{equation}\label{equ:conc3}
\pi_{\xb}(\tvu \times (\tilde{\nabla} \times \tvu))
= \vu \times ( (\xb \cdot (\tilde{\nabla} \times \tvu)) \xb),\;\; \xb\in \S.
\end{equation}

So, for a \dela{tangential} vector field $\vu$ on $\S$, one can define
\begin{equation}\label{equ:defcurln}
\Curln \vu := \xh \cdot (\tilde{\nabla} \times \tvu)_{\vert\S}.
\end{equation}
Given a {tangential} vector field $\vv$ on $\S$, with a slight abuse
of notation, we write
\[
 \vv \times \Curln \vu :=  \vv \times \xb (\Curln \vu).
\]
Hence from \eqref{equ:conc3}, \eqref{equ:defcurln}
we obtain
\[
  \pi_{\xb} [\tvu \times (\tilde{\nabla} \times \tvu)] (\xb)=  [\vu(\xb) \times \xb] {\Curln}\vu(\xb) \,   \;\; \xb \in \mathbb{S}^2, \label{eqn:star}
\]
and thus
\begin{equation}\label{short_nonlin1}
        \nnabla_{\!\!\!\vu} \vu = \Grad \frac{|\vu|^2}{2} - \vu \times {\Curln}\vu.
\end{equation}

We have the following well-defined operators \cite{Ili91}:
\begin{definition}
Let $\vu$ be a tangent vector field on $\S$, and let the vector field $\bpsi$ be normal to $\S$.
We set
\begin{equation}\lb{Curl}
\Curln \vu = (\xh \cdot (\tilde{\nabla} \times \tvu))|_{\S}, \qquad
\Curl {\bf \psi} = (\tilde{\nabla} \times \tilde{\bpsi})|_{T\S},
\end{equation}
where (as before) the vector field $\tvu$ \dela{(which is called an extension of $\vu$)} is defined in some
neighbourhood of $\S$ in $\R^3$  and satisfies $\tvu|_{\S} = \vu$.
\end{definition}
The first equation in \eqref{Curl} indicates a projection of
$\nabla \times \tvu$ onto the normal direction
while the second equation means a restriction of $\nabla \times \bpsi$
to the tangent field on $\S$. The definitions in~\eqref{Curl} are independent
of the extensions $\tvu$ and $\tilde{\bpsi}$.
A vector field $\bpsi$ normal to $\S$ will often be
identified with a scalar function on $\S$ when it is convenient to do so.
The relationships among $\Curl$ of a scalar function $\psi$, $\Curl$ of
a normal vector field $\vw =  w \xh$, of a scalar function $v$,
and $\Curln$ of a \dela{tangential} vector field $\vv$ on $\S$ and the surface
$\Div$ and $\Grad$ operators are given as (see \cite{CaoRamTit99})
\begin{equation}\label{Curl2}
\Curl \psi = -\xh \times \Grad \psi,  \quad
\Curl \vw = -\xh \times \Grad w, \quad
\Curln \vv = -\Div(\xh \times \vv).
\end{equation}

The surface diffusion operator acting on  \dela{tangential}
vector fields  on $\S$ is  denoted by $\DDelta$
(known as the vector Laplace-Beltrami or Laplace-de Rham operator)
and  is defined as
\begin{equation}\label{deRham}
 \DDelta \vv = \Grad \Div \vv - \Curl \Curln \vv.
\end{equation}

Using \eqref{Curl2}, one can derive the following relations connecting the above operators:
\begin{equation} \label{curln_curl_psi}
  \Div \Curl v = 0, \qquad
  \Curln \Curl v = -\xh \Delta v, \qquad
\DDelta \Curl v = \Curl \Delta v.
\end{equation}

We introduce the standard inner products on
the space $L^2(\S)$ of square integrable scalar functions
on $\S$, and on the space  $\lL^2(\S)$ of \dela{tangential} vector
fields on $\S$, denoted by :
\bea
(v_1,\;v_2) &=& (v_1,\;v_2)_{L^2(\S)} ~~= \int_{S} v_1 v_2~dS,
\quad \qquad v_2, v_2 \in L^2(\S), \label{scal_ip} \\
(\vv_1,\;\vv_2)   &=& (\vv_1,\;\vv_2)_{\lL^2(S)} = \int_{\S} \vv_1 \cdot
\vv_2~dS, \qquad  \vu,\vv \in \lL^2(\S), \label{vec_ip}
\eea
where $\int_{\S} v dS$ denotes integration with respect to
the surface (or riemannian volume) measure on $\S$. In the spherical coordinates we have,
locally, $dS = \sin \theta d\theta d\phi$.
Throughout the paper, the  induced norm on  $\lL^2(\S)$ is
denoted by $\| \cdot \|$ and for other inner product spaces,
say $X$ with inner product $(\cdot,\;\cdot )_{X}$, the associated norm is
denoted by  $\| \cdot \|_X$.

We have the following identities for appropriate real valued functions and
vector fields on $\S$, see for instance \cite[(2.4)-(2.6)]{Ili91}.
\begin{eqnarray}
(\Grad \psi,\; \vv) &=& - (\psi,\; \Div \vv), \label{ip_identities0} \\
(\Curl \psi,\;\vv)  &=& (\psi,\; \Curln \vv), \label{ip_identities1} \\
(\Curl \Curln \vw,\; \ovz) &=& (\Curln \vw,\; \Curln \ovz).
  \label{ip_identities2}
\end{eqnarray}
In \eqref{ip_identities1}, the  $\lL^2(\S)$ inner
product is used on the left hand side and the $L^2(\S)$ inner product
is used on the right hand side. Throughout the paper, we identify a
normal vector field $\vw$ with a
scalar field $w$ by $\vw = \xh w$ and hence we put
\begin{equation}\label{normal_ip}
    (\psi,\; \vw) := (\psi,\; w)_{L^2(\S)}, \qquad
    \vw = \xh w, \qquad \psi, w \in L^2(\S).
\end{equation}

Using the identity \eqref{ip_identities0} the unknown pressure can be
eliminated from the
first equation in (\ref{NSE_sphere}) through the weak formulation.

\subsection{The weak formulation}
We now introduce Sobolev spaces $H^s(\S)$ and $\hH^s(\S)$ of scalar functions and
vector fields on $\S$ respectively.

Let $\psi$ be a scalar function and let $\vu$ be a vector field on $\S$, respectively.
For $s\ge 0$ we define
\begin{equation}\label{def:Hs}
  \|\psi\|^2_{H^s(\S)} = \|\psi\|^2_{L^2(\S)} + \| (-\Delta)^{s/2} \psi\|^2_{L^2(\S)},
\end{equation}
and
\begin{equation}\label{def:hHs}
  \|\vu\|^2_{\hH^s(\S)} = \|\vu\|^2 + \|(-\DDelta)^{s/2} \vu\|^2,
\end{equation}
where $\Delta$ is the Laplace--Betrami and $\DDelta$ is the Laplace--de Rham
operator on the sphere. In particular, for $s=1$,
\begin{align}
\|\vu\|^2_{\hH^1(\S)}
     &= \|\vu\|^2 + (\vu,-\DDelta \vu)  \nonumber\\
     &= \|\vu\|^2 + \|\Div \vu\|^2 + \| \Curln \vu\|^2, \label{equ:H1norm}
\end{align}
where we have used formulas \eqref{deRham},\eqref{ip_identities0}--\eqref{ip_identities2}.

By the Hodge decomposition theorem~\cite[Theorem 1.72]{Aub98} the space of $C^\infty$ smooth  \dela{tangential} fields on $\S$
can be decomposed into three components
:
\begin{equation} \label{Hodge}
  C^\infty(T\S) =   \calG \oplus  \calV \oplus \calH,
\end{equation}
where
\begin{equation} \label{orth_space}
   \calG = \{\Grad \psi: \psi \in C^\infty(\S)\},\quad
   \calV = \{\Curl \psi: \psi \in C^\infty(\S)\},
\end{equation}
and $\calH$ is the finite-dimensional space of harmonic fields. Since
the sphere is simply connected, $\calH = \{0\}$.
We introduce the following spaces
\beas
    H &=& \mbox{ closure of } \calV
          \mbox{ in } \lL^2(\S), \\
    V &=& \mbox{ closure of } \calV
          \mbox{ in } \hH^1(\S).
\eeas
Note that
\beas
    H &=& \{\vu \in  \lL^2(\S): \Div \vu =0\}, \\
    V &=& H\cap  \hH^1(\S).
\eeas




Since $V$ is densely and continuously embedded into $H$ and $H$ can be identified
with its dual $H^\prime$, and denoting by  $V^\prime$  the dual of $V$, we have the following Gelfand triple:
\begin{equation}\label{Gelfand triple}
 V \subset H \cong H^\prime \subset V^\prime.
\end{equation}

We consider the following linear Stokes problem, i.e. given $\vf\in V^\prime$, find $\vv \in V$ such that
\begin{equation}\label{Stokes}
  \nu \Curl \Curln \vu - 2\nu{\Ric}  (\vu) + \Grad p = \vf, \quad \Div \vu = 0.
\end{equation}
Taking the inner product of the first equation of \eqref{Stokes} with a test vector field $\vv\in V$
and then using \eqref{ip_identities2}, we obtain
\begin{equation}\label{weakStokes}
 \nu (\Curln\vu,\Curln\vv) - 2\nu(\mbox{Ric\;}\vu,\vv) = (\vf,\vv), \;\;\; \vv \in V.
\end{equation}
We define a bilinear form $a: V\times V \rightarrow \R$ by
\[
  a(\vu,\vv) :=
  (\Curln\vu,\Curln\vv) - 2(\mbox{Ric\;}\vu,\vv),\quad \vu,\vv \in V.
\]
In view of \eqref{equ:H1norm} and \eqref{equ:RicS}, the bilinear form
$a$ satisfies
\[
  a(\vu,\vv) \le  \|\vu\|_{\hH^1} \|\vv\|_{\hH^1},
\]
and hence it is continuous on $V$. So by the Riesz Lemma, there exists a unique
operator $\calA:~V \rightarrow V^\prime$,
 such that $a(\vu,\vv) = (\calA \vu,\vv)$, for
$\vu, \vv \in V$. Using the Poincar\'{e} inequality,
we also have $a(\vu,\vu) \ge \alpha \|\vu\|^2_V$, for some
constant $\alpha>0$, which means $a$
is coercive in $V$. Hence by the Lax-Milgram
theorem the operator $\calA : V \rightarrow V^\prime$ is an isomorphism. Furthermore,
by using \cite[Theorem 2.2.3]{Tan79}, we conclude that the operator $\calA$ is
positive definite, self-adjoint and $\calD(\calA^{1/2}) = V$. Thus, the spectrum
of $\calA$ consists of an infinite sequence
$0 < \lambda_1 < \lambda_2 < \ldots < \lambda_\ell \rightarrow~\infty$ of eigenvalues (of finite multiplicity), and
there exists an orthogonal basis $\{\vw_\ell\}_{\ell\ge 1}$ of $H$ consisting
of eigenvectors of $\calA$.

\deld{\textcolor{red}{The citation \cite{Mul66} below is missing.}}
By using the stream function $\psi_\ell$ for which $\vw_\ell = \Curl \psi_\ell$
and identities \eqref{curln_curl_psi} we can show that $\lambda_\ell$ is an
eigenvalue of the Laplace-Beltrami operator $\Delta$. It is well known
that, see \cite{Mul66}, these eigenvalues are $\lambda_\ell=\ell(\ell+1)$ for
$\ell=0,1,\ldots$ and that the eigenfunctions are the spherical harmonics
$Y_{\ell,m}(\theta,\phi)$, for $\theta \in [0,\pi]$,
$\phi \in [0,2\pi)$, which are defined by (in spherical coordinates)
\begin{equation}\label{sph_har}
   Y_{\ell,m}(\theta,\varphi) =
      \left[\frac{(2\ell+1)}{4\pi}
      \frac{(\ell-|m|)!}{(\ell+|m|)!} \right]^{1/2} P^{m}_{\ell}(\cos \theta) e^{im\varphi},
      \quad m=-\ell, \ldots, \ell,
\end{equation}
with $P^m_\ell$ being the associated Legendre polynomials.

Hence for each positive integer $\ell=1,2,\ldots$,
the eigenvectors of the operator $\calA$
corresponding to the eigenvalue $\lambda_\ell$
are given by
\begin{equation}\label{stok_eig}
 \qquad     \Z_{\ell,m}(\theta,\varphi),
 \quad m=-\ell, \ldots, \ell.
\end{equation}

Since $\{\Z_{\ell,m}: \ell=1,\ldots; m =-\ell,\ldots, \ell\}$ is an orthonormal
basis for $H$, an arbitrary $\vv \in H$ can be written as
\begin{equation}\label{Fourier_coeff}
\vv=\sum_{\ell=1}^\infty \sum_{m=-\ell}^\ell \widehat{v}_{\ell,m} \Z_{\ell,m},
\qquad
\widehat{v}_{\ell,m} = \int_{\S} \vv \cdot \overline{\Z_{\ell,m}} dS =
(\vv,\;\Z_{\ell,m}).
\end{equation}

Next we define an operator $\A$ in $H$ as follows:
\begin{equation}\label{defA}
\left\{ \begin{array}{lcl}
            \calD(\A) &:=& \{\vu \in V: \calA\vu \in H \}, \\
            \A \vu    &:=& \calA \vu, \quad \vu \in \calD(\A).
        \end{array} \right.
\end{equation}

Let ${\mathrm P}$ be the Leray-Helmholtz orthogonal projection from $\lL^2(\S)$ to $H$.
It can be shown \cite{Gri00} that $\calD(\A) = \hH^2(\S) \cap V$,
$\A = -{\mathrm P}(\DDelta + 2{\Ric} )$, and $\A$ is self-adjoint in $H$.
It can also be show that
$V=\calD(\A^{1/2})$ and
\[
 \|\vu\|^2_V \sim (\A \vu,\vu), \quad \vu \in \calD(\A),
\]
where $A \sim B$ indicates that there are two positive constants $c_1$ and $c_2$
such that $c_1 A \le B \le c_2 A$.

Let us now recall the definition of the fractional power $\A^{s/2}$ of the Stokes operator $\A$ in $H$.  For any $s\ge 0$ its domain is given by
\begin{equation}\label{Gev0_space}
 \calD(\A^{s/2}) =
 \left\{\vv \in  H \; : \vv = \sum_{\ell=1}^\infty
       \sum_{m=-\ell}^{\ell} \widehat{v}_{\ell,m} \Z_{\ell,m},
     \quad\sum_{\ell=1}^\infty
     \sum_{m=-\ell}^\ell \lambda^s_\ell |\widehat{v}_{\ell,m}|^2<\infty
     \right\},
\end{equation}
which is equal to the set of   divergence-free vector fields from  the Sobolev space $\hH^s(\S)$.
For $\vv\in \calD(\A^{s/2})$, we define
\begin{equation}\label{As-def}
\A^{s/2} \vv :=
\quad\sum_{\ell=1}^\infty
  \sum_{m=-\ell}^\ell \lambda^{s/2}_\ell \widehat{\vv}_{\ell,m} \Z_{\ell,m} \quad \in H.
\end{equation}

The Coriolis operator $\CC_1:\lL^2(\S)\rightarrow \lL^2(\S)$, is defined by the formula
\begin{equation}\label{C}
(\CC_1\vv)(\xb) = 2 \Omega \cos \theta (\xb \times \vv(\xb)),\quad \xb \in \S,
\end{equation}
From the definition, it can be seen that $\CC_1$ is a bounded linear operator defined on
$\lL^2(\S)$. In the sequel we will need the operator $\CC ={\mathrm P} \CC_1$ which is
well defined and bounded in $H$.
Furthermore, for $\vu \in H$
\be\lb{Cuu}
  (\CC\vu, \vu) = (\CC_1\vu, {\mathrm P}\vu)=
  \int_{\S} 2 \Omega \cos \theta(\xb) ((\xb \times \vu) \cdot \vu(\xb))dS(\xb) = 0.
\ee

We consider the trilinear form $b$ on  $V \times  V \times V$,
defined as
\begin{equation}\label{trilinear}
b(\vv,\vw,\ovz) = (\nnabla_{\!\!\!\vv} \vw,  \ovz) =  \int_{\S} \nnabla_{\!\!\!\vv} \vw \cdot
\ovz~dS, \qquad \vv, \vw, \ovz \in  V.
\end{equation}
Using the following identity
\begin{eqnarray}\label{covar_long}
       2 \nnabla_{\!\!\!\vw} \vv &=& -\Curln(\vw\times \vv)+\Grad(\vw\cdot \vv)-\vv\Div \vw \\
                          &+&\vw \Div \vv - \vv\times\Curln \vw - \vw \times \Curln \vv.
                          \nonumber
\end{eqnarray}
and (\ref{ip_identities1}), for divergence free
fields $\vv, \vw, \ovz$,  the trilinear
form can be written as
\be
b(\vv,\vw,\ovz) =  \frac{1}{2} \int_{\S} \left[-\vv \times \vw \cdot \Curln \ovz  +
         \Curln \vv \times \vw \cdot \ovz - \vv \times \Curln \vw \cdot \ovz\right]~dS.
     \label{short_b}
\ee
Moreover~\cite[Lemma 2.1]{Ili91}
\be \lb{skew}
   b(\vv,\vw,\vw)=0, \qquad b(\vv,\ovz,\vw) = -b(\vv,\vw,\ovz)
          \qquad \vv \in V, \vw, \ovz \in \hH^1(\S).
\ee

We have the following inequality from \cite{IliFil89}
\be\lb{equ:L4}
 \|\vu\|_{\lL^4(\S)} \le C \|\vu\|^{1/2}_{\lL^2(\S)} \|\vu\|_V^{1/2},
 \quad \vu \in \hH^1(\S).
\ee

Thus, using (\ref{deRham}),  (\ref{ip_identities1}), (\ref{defA}), and (\ref{short_b}),
a \emph{weak solution} of the Navier-Stokes equations (\ref{NSE_sphere}) is a vector field
$\vu \in L^2([0,T];V)$ with $\vu(0)= \vu_0$ that satisfies the weak form of equation~\eqref{NSE_sphere}, i.e.
\begin{equation} \label{weak_form}
  (\partial_t \vu,\vv) + b(\vu,\vu,\vv) + \nu(\Curln \vu, \Curln \vv) - 2\nu (\mbox{Ric\;}\vu,\vv)
  + (\CC\vu, \vv)  = (\vf,\vv), \qquad  \vv \in  V.
\end{equation}
This weak formulation can be written in operator equation form on $V^\prime$, the dual of $V$.
Let $\vf \in L^2([0,T];V^\prime)$ and $\vu_0 \in H$.
We want to find a vector field  $\vu \in L^2([0,T]; V)$, with
$\partial_t \vu \in L^2([0,T];V^\prime)$
such that
\begin{equation}\label{op_form}
   \partial_t \vu + \nu \A\vu + \B(\vu,\vu) + \CC\vu = \vf, \qquad \vu(0) = \vu_0,
\end{equation}
where the bilinear form $\B:V \times V \rightarrow V^\prime$ is defined by
\begin{equation}\label{B}
   (\B(\vu,\vv),\vw) = b(\vu,\vv,\vw) \qquad  \vw \in V.
\end{equation}
With a slight abuse of notation, we also denote $\B(\vu) = \B(\vu,\vu)$.


The following are some fundamental properties of the trilinear form $b$;
see \cite{FenFre05}: There exists a constant $C>0$ such that
\be\label{b_estimate}
|b(\vu,\vv,\vw)| \le C
 \begin{cases}
   \|\vu\|^{1/2} \|\vu\|^{1/2}_V \|\vv\|^{1/2}_V \|\A\vv\|^{1/2} \|\vw\|,
     \quad \vu \in V, \vv\in D(\A), \vw \in H,\\
     \\
   \|\vu\|^{1/2} \|\A \vu\|^{1/2} \|\vv\|_V \|\vw\|,
   \quad \vu \in D(\A), \vv \in V, \vw\in H, \\
   \\
   \|\vu\|^{1/2} \|\vu\|^{1/2}_V \|\vv\|_V \|\vw\|^{1/2} \|\vw\|_V^{1/2},
   \quad \vu,\vv,\vw \in V.
 \end{cases}
\ee

We also need the following estimates:
\begin{lemma}
There exists $C>0$ such that
\be\lb{b_estimate1}
|b(\vu,\vv,\vw)| \le C \|\vu\| \|\vw\|
 (\|\Curln \vv\|_{\lL^\infty(\S)} + \|\vv\|_{\lL^\infty(\S)}),
\quad \vu \in H, \vv \in V, \vv \in H,
\ee
and
\be\lb{b_estimate2}
|b(\vu,\vv,\vw)| \le C \|\vu\| \|\vv\|_V \|\vw\|^{1/2} \|\A\vw\|^{1/2},
\quad \vu \in H, \vv \in V, \vw \in D(\A),
\ee
and
\be\lb{b_estimate3}
|b(\vu,\vv,\vw)| \le C \|\vu\|_{\lL^4(\S)} \| \vv \|_V \|\vw\|_{\lL^4(\S)},
\quad \vv \in V, \vu,\vw \in \hH^1(\S).
\ee
\end{lemma}
\begin{proof}
Following~\cite[Page~574]{Ili91}, for  $\vu \in C^\infty(T\S)$,
we extend $\vu$ to the spherical layer
$\Omega = \S \times I$, $I = (r_1,r_2), 0<r_1 < 1 < r_2 < \infty$
by the formula
\be\label{ext_map}
 \tilde{\vu}(\x) = \varphi(|\x|)\vu(\x / |\x|), \quad \x \in \Omega,
\ee
where $\varphi \in C_0^\infty(I)$,   $\varphi(t) \ge 0,~t \in I$,
and  $\varphi(1) = 1$. Since
\[
   \int_{\S\times I} |\tvu|^p =
     \int_{r_1}^{r_2} \varphi^p(r) dr \int_{\S} |\vu|^p dS,
\]
there exists $c_p=c_p(\varphi,r_1,r_2)$ such that
\be\label{Lpnorm}
     \|\tvu\|_{\lL^p(\S\times I)} = c_p \|\vu\|_{\lL^p(\S)}.
\ee


Using identities (4.24)--(4.26) and inequality (3.7) from \cite{Ili91} we can also deduce
that there exists $c>0$ such that for every tangent vector field $\vu\in V$,
\be\label{Vnorm_ext}
 \|\nabla \tilde{\vu} \|_{\lL^2(\Omega)} \le c \|\vu\|_V,
\ee
and
\be\label{Sup_ext}
 \|\nabla \tilde{\vu}\|_{\lL^\infty(\Omega)} \le c (\|\Curln \vu\|_{\lL^\infty(\S)} + \|\vu\|_{\lL^\infty(\S)}).
\ee

Suppose $\vu,\vv,\vw$ are extended from $\S$ to the spherical
layer $\S \times I$ by \eqref{ext_map}. Then from
\cite[Lemma 4.3]{Ili91} we have
\be\label{btilde}
    b(\vu,\vv,\vw) = c \tilde{b}(\tilde{\vu},\tilde{\vv},\tilde{\vw}),
\ee
where $\tilde{b}$ is the extenstion of $b$ to $\Omega$, i.e.,
\[
  \tilde{b}(\tilde{\vu},\tilde{\vv},\tilde{\vw})
   = \sum_{i,j = 1}^{3}\int_{\Omega}
       \tilde{u}_j \frac{\partial \tilde{v}_i}{\partial x_j} \tilde{w}_i dx.
\]

Since by  the Cauchy-Schwarz inequality
\[
|\tilde{b}(\tvu,\tvv,\tvw)| =
\left |\sum_{i,j = 1}^{3}\int_{\Omega}
    \tilde{u}_j \frac{\partial \tilde{v}_i}{\partial x_j} \tilde{w}_i dx
             \right|
\le \|\tvu\| \|\nabla \tvv \|_{\lL^\infty(\S)} \|\tvw\|,
\]
by restricting to the sphere by \eqref{btilde} and using \eqref{Sup_ext} we obtain
\eqref{b_estimate1}.

On the sphere $\S$, we have the following version of Sobolev embedding
inequality \cite[Proposition 2.11]{Aub98}
\be\label{Sobolev_imbedding}
  \| \vu \|_{\lL^q(\S)} \le C \|\vu\|_{\hH^{s}(\S)},
         \qquad s<1, \quad \frac{1}{q}=\frac{1}{2} - \frac{s}{2}.
\ee
By the Cauchy-Schwarz inequality, we have
\be\label{b_extended}
|\tilde{b}(\tvu,\tvv,\tvw)| =
\left |\sum_{i,j = 1}^{3}\int_{\Omega}
    \tilde{u}_j \frac{\partial \tilde{v}_i}{\partial x_j} \tilde{w}_i dx
         \right|
             \le \|\tvu\|_{\lL^{2}(\Omega)} \| \nabla \tvv \|_{\lL^{2}(\Omega)}
                       \|\tvw\|_{\lL^\infty(\Omega)}
\ee
Restricting to the sphere by \eqref{btilde} and using \eqref{Vnorm_ext} we obtain
\[
 |b(\vu,\vv,\vw)| \le C \|\vu\| \|\vv\|_V \|\vw\|_{\lL^\infty(\S)}.
\]
By the Sobolev imbedding theorem \eqref{Sobolev_imbedding} and
the H\"{o}lder inequality for Riemannian manifolds (see \cite[Proposition 3.62]{Aub98})
we have
\[
 \|\vw\|_{\lL^\infty(\S)} \le \|\vw\|^{1/2} \|\A\vw\|^{1/2}.
\]
Combining those inequalities, we arrive at inequality \eqref{b_estimate2}.

By the H\"{o}lder inequality, we obtain
\[
|\tilde{b}(\tvu,\tvv,\tvw)| =
\left |\sum_{i,j = 1}^{3}\int_{\Omega}
    \tilde{u}_j \frac{\partial \tilde{v}_i}{\partial x_j} \tilde{w}_i
             \right|
    \le \|\tvu\|_{\lL^{4}(\Omega)} \| \nabla \tvv \|_{\lL^{2}(\Omega)}
                                                 \|\tvw\|_{\lL^4(\Omega)}.
\]
Restricting to the sphere by \eqref{btilde} and using \eqref{Vnorm_ext}, we obtain \eqref{b_estimate3}.
\end{proof}

In view of \eqref{b_estimate3}, $b$ is a bounded trilinear map from
$\lL^4(\S) \times V \times \lL^4(\S)$ to $\R$. Moreover, we have the following
result:

\begin{lemma}\label{lem:bL4}
The trilinear map $b:V \times V \times V \rightarrow \R$ has a unique
extension to a bounded trilinear map from $\left( \lL^4(\S) \cap H\right) \times \lL^4(\S) \times V$
to $\R$.
\end{lemma}

It follows from \eqref{b_estimate3} that $b$ is a bounded trilinear map
from $\lL^4(\S)~\times~V~\times~\lL^4(\S)$ to $\R$. It follows that $\B$ maps
$\lL^4(S) \cap H$ (and so $V$) into $V^\prime$ and
\be\lb{equ:Bu L4}
 \|\B(\vu)\|_{V^\prime} \le C_1 \|\vu\|^2_{\lL^4(\S)} \le  C_2 \|\vu\| \|\vu\|_V
                 \le C_3 \|\vu\|^2_V, \quad \vu \in V.
\ee
\section{Stochastic Navier--Stokes equation on a rotating unit sphere}
By adding a white noise term to \eqref{NSE_sphere}, we obtain the stochastic NSE on the sphere, which is the main object of our analysis:
is
\begin{equation}\label{snss}
\partial_t \vu + \nabla_{\vu} \vu - \nu \bL \vu + \vomega\times \vu +\Grad p = \vf + n(\x,t),
\quad \Div \vu = 0, \quad \vu(\x,0) = \vu_0.
\end{equation}
We assume that $\vu_0 \in H$, $\vf \in V^\prime$ and $n(x,t)$ is a Gaussian random
field which is a white noise in time. By applying  the Leray-Helmholtz projection  we can interpret equation \eqref{snss} as the following abstract   It\^o equation in $H$
\begin{equation}\label{sNSEs}
 d\vu(t) + \A \vu(t) dt + \B(\vu(t),\vu(t)) dt + \CC \vu = \vf dt + G dW(t),
 \quad \vu(0) = \vu_0.
 \end{equation}
Here $\vf$ is the deterministic forcing term and $\vu_0$ is the initial velocity.

We assume that $W$ is a cylindrical Wiener process \cite{DaPZab92} defined
on a  probability space $(\Omega,\mathcal{F},\bbP)$.
$G$ is a linear continuous operator,
which determines the spatial smoothness of the noise term, satisfying
further assumptions to be specified later.

Roughly speaking, a solution to problem \eqref{sNSEs} is a process $\vu(t)$, $t\ge 0$,
which can be represented in the form $\vu(t) = \vv(t) + \ovz_\alpha(t)$, where
$\ovz_\alpha(t)$, $t\in \R$, is a stationary Ornstein--Uhlenbeck process with
drift $-\nu\A - \CC -\alpha I$, i.e. a stationary solution of
\be\label{Orn Uhl}
 d\ovz_{\alpha} + (\nu\A+\CC +\alpha)\ovz_{\alpha} dt = G dW(t), \quad t \in \R,
\ee
and $\vv(t)$, $t \ge 0$, is the solution to the following problem
(with $\vv_0 = \vu_0 - \ovz_{\alpha}(0)$):
\begin{equation}\label{equ:ode}
\left\{\begin{array}{l}
 \partial_t \vv = -\nu \A\vv - \B(\vv+\ovz_{\alpha},\vv+\ovz_{\alpha}) - \CC \vv +
 \alpha \ovz_{\alpha} + \vf, \\
 \vv(0) = \vv_0.
\end{array}\right.
\end{equation}

\begin{definition}
Suppose that
$\ovz \in L_{\loc}^4([0,\infty);\lL^4(\S))$, $\vf \in V^\prime$
and $\vv_0 \in H$. A function
$\vv \in C([0,\infty);H) \cap L^2_{\loc}([0,\infty);V)$
is a solution to problem \eqref{equ:ode} if and only if
$\vv(0) = \vv_0$ and \eqref{equ:ode} holds in the weak sense, i.e. for any $\phi \in V$,
\be\label{def:soln}
\partial_t (\vv,\phi) = -\nu(\vv,\A\phi) - b(\vv+\ovz,\vv+\ovz,\phi)
    - (\CC \vv,\phi) + (\alpha \ovz + \vf, \phi).
\ee
\end{definition}

We remark that for \eqref{def:soln} to make sense, it is sufficient to
assume that $\vv \in L^2(0,T;V) \cap L^\infty(0,T;H)$.

\deld{{\coma{Below we assumed that $\ovz \in L^4_{\loc}([0,\infty);\lL^4(\S))
\cap L^2_{\loc}([0,\infty);V)$. I think this was an error and we don't need the second fact. But Please check!}
}}
\begin{theorem}\label{thm:existence}
Assume that $\alpha \ge 0$,
$\ovz \in L^4_{\loc}([0,\infty);\lL^4(\S) \cap H)$,
$\vv_0~\in H$ and~$\vf\in V^\prime$. Then then there
exists a unique solution $\vv$ of problem \eqref{equ:ode}.
\end{theorem}

\deld{{\coma{I removed a factor $\cap L^2(0,T;V^\prime),$ from below.}}}
\begin{theorem}\label{thm:limit}
Assume that $\vu_{0n} \rightarrow \vu_0$ in $H$, and, for some fixed $T>0$,
\[
\ovz_n \rightarrow \ovz \mbox{ in } L^4([0,T];\lL^4(\S) \cap H)
\mbox{ and }
\vf_n \rightarrow \vf \mbox{ in\,\,} L^2(0,T;V^\prime).
\]
Let us denote by $\vv(t,\ovz)\vu_0$ the solution of problem \eqref{equ:ode}
and by $\vv(t,\ovz_n) \vu_{0n}$ the solution of problem \eqref{equ:ode} with
$\ovz, \vf, \vu_0$ being replaced by $\ovz_n, \vf_n, \vu_{0n}$. Then
\[
 \vv(\cdot, \ovz_n) \vu_{0n} \rightarrow \vv(\cdot,\ovz) \vu_0 \mbox{ in } C([0,T];H) \cap L^2(0,T;V).
\]
In particular, $\vv(T,\ovz_n)\vu_{0n} \rightarrow \vv(T,\ovz)\vu_0$ in $H$.
\end{theorem}
\section{Proofs of Theorems \ref{thm:existence} and \ref{thm:limit}}

\subsection{Proof of Theorem \ref{thm:existence}}
For the proof, we need the following classical result, see \cite[Lemma III.1.2]{Tem79}.

\begin{lemma}\label{lem:uprime}
Suppose that $ V \subset H \cong H^\prime \subset V^\prime$ is a Gelfand triple of Hilbert spaces.
If a function $\vu$ belongs to $L^2(0,T;V)$ and its weak derivative belongs to $L^2(0,T;V^\prime)$,
then $\vu$ is a.e. equal to a continuous function from $[0,T]$ to $H$, the real-valued function
$\|\vu\|^2$ is absolutely continuous and, in the weak sense of $(0,T)$, one has
(with $\langle\cdot,\cdot\rangle$ being the duality between $V^\prime$ and $V$)
\be\lb{equ:du2}
 \partial_t \|\vu(t)\|^2 = 2 \langle \partial_t \vu(t), \vu(t) \rangle.
\ee
\end{lemma}

Let us assume that $X_0 \subset X \subset X_1$ are Hilbert spaces with the
injections being continuous and the injection of $X_0$ into $X$ is compact.

If $v$ is a function from $\R$ to $X_1$, we denote by $\hat{v}$ its Fourier
transform
\[
\hat{v}(\tau) = \int_{\R} e^{-2i\pi t \tau} v(t) dt.
\]
The fractional derivative in $t$ of order $\gamma$ of $v$ is the inverse Fourier transform
of the $X_1$-valued function $\{ \mathbb{R}\ni \tau \mapsto (2i \pi \tau)^{\gamma} \hat{v}(\tau)\}$, i.e.
\[\widehat{D^\gamma_t v\,}(\tau) = (2i \pi \tau)^{\gamma} \hat{v}(\tau),\;\;\; \tau \in \mathbb{R}.\]

For a given $\gamma >0$, we define the space
\[
 \calH^{\gamma,2}(\R;X_0,X_1) = \{ v \in L^2(\R;X_0): D^\gamma_t v \in L^2(\R;X_1)\}.
\]
This is a Hilbert space equipped with the norm
\[
\|v\|_{ \calH^{\gamma,2}(\R;X_0,X_1) } =
  \big(\|v\|^2_{L^2(\R;X_0)} + \| |\tau|^\gamma \hat{v}\|_{L^2(\R;X_1)} \big)^{1/2}
\]
\noindent
For a given set $K \subset \R$, the subspace $\calH^{\gamma,2}_K$ of $\calH^{\gamma,2}=\calH^{\gamma,2}(\R;X_0,X_1)$
is defined by
\[
 \calH^{\gamma,2}_K(\R;X_0,X_1) = \{ u \in \calH^\gamma(\R;X_0,X_1), \mbox{supp } u \subset K\}.
\]

The compactness theorem (\cite[Theorem III.2.2]{Tem79}) is stated as follows

\begin{theorem}\label{thm:compact}
Suppose that $X_0 \subset X  \subset X_1$ is a Gelfand triple of Hilbert spaces and
the injection of $X_0$ into $X$ is compact. Then for any bounded set $K$ and $\gamma>0$,
the injection of $\calH^\gamma_K(\R;X_0,X_1)$ into $L^2(\R,X)$ is compact.
\end{theorem}

\noindent
{\bf Part I. Existence of solutions}

Let $H_L = \mbox{lin span} \{\Z_{\ell,m} : \ell=1,\ldots,L; |m| \le L \}$ with the
norm inherited from $H$ and $V_L = \mbox{lin span} \{\Z_{\ell,m} : \ell=1,\ldots,L; |m| \le L \}$ with
the norm inherited from $V$. We denote by $P_L$ the orthogonal projection from $H$
onto $H_L$. We consider the following approximate problem for \eqref{equ:ode}
on the finite dimensional space $H_L$:
\begin{equation}\label{equ:Galerkin}
\left\{ \begin{array}{ccl}
         \partial_t \vv_L &=& P_L \big[ - \nu \A\vv_L
         - \B(\vv_L) - \B(\vv_L,\ovz) - \B(\ovz,\vv_L)
         - \CC \vv_L + F \big], \\
         \vv_L(0)  &=& P_L \vv_0,
         \end{array}
\right.
\end{equation}
where $F = -\B(\ovz) + \alpha \ovz + \vf$.

From the condition $\ovz \in L^4_{\loc}([0,\infty);\lL^4(\S) \cap H)$
and Lemma~\ref{lem:bL4} we conclude
that $F \in L^2(0,T;V')$.

\dela{ If we express $\vv_L$ as $$\vv_L=\sum_{\ell=1}^L \sum_{m=-L}^L g_{L;\ell,m}(t) \Z_{\ell,m}$$
then by taking the inner product of \eqref{equ:Galerkin} with each eigenfunction $\Z_{\ell,m}$
\bea
(\partial_t \vv_L(t), \Z_{\ell,m}) &=&
   -\nu (\A \vv_L,\Z_{\ell,m}) - (\B(\vv_L,\vv_L),\Z_{\ell,m})
   -(\B(\vv_L,\ovz),\Z_{\ell,m}) \nonumber \\
    && - (\B(\ovz,\vv_L),\Z_{\ell,m}) - (\CC \vv_L,\Z_{\ell,m}) + (F,\Z_{\ell,m}) \label{equ:pre ode} \\
&&t \in [0,T], \;\; \ell=1,\ldots,L; \;\; m = -L,\ldots, L,  \nonumber
\eea
we can transform}
The equation  \eqref{equ:Galerkin} can written in  the following equivalent form
\begin{equation}\lb{equ:odesys}
\left\{ \begin{array}{ccl}
  \frac{d \vv_L}{dt} &=& G(t,\vv_L(t)), \quad t \ge 0, \\
  \vv_L(0) &=& P_L \vv_0,
\end{array}
\right.
\end{equation}
where the function $G(t,\vv)$ is
a function locally Lipschitz with respect to $\vv$ and measurable with respect to $t$.

By the local existence and uniqueness theorem for ordinary differential equations,
the system of nonlinear differential equations \eqref{equ:odesys} has a maximal solution
defined on some interval $[0,T_L)$. If $T_L < \infty$, then $\|\vu_L(t)\|$ must tend to $+\infty$
as $t \rightarrow t_L$; the a priori estimates we shall prove later show that this does not
happen and therefore $t_L=\infty$.

We observe that
\[
 -\nu (P_L \A \vv_L,\vv_L) = -\nu(\A \vv_L,\vv_L) = -\nu \|\vv_L\|^2_V
\]
and by \eqref{skew}
\[
 (P_L\B(\vv_L,\vv_L) = (\B(\vv_L),\vv_L) = 0, \quad
 (P_L \B(\ovz,\vv_L),\vv_L) = (\B(\ovz,\vv_L),\vv_L) = 0,
\]
and by \eqref{Cuu},
\[
 (P_L \CC \vv_L,\vv_L) = (\CC \vv_L,\vv_L) = 0.
\]
Using \eqref{b_estimate3},\eqref{equ:L4} and the Young
inequality ($ab \le a^p/p+b^q/q$ with $p=4$ and $q=4/3$) we have
\bea
|b(\vv_L,\vv_L,\ovz)| &\le & C \|\vv_L\|_{\lL^4(\S)} \|\vv_L\|_V \|\ovz\|_{\lL^4(\S)}
            \nonumber\\
             &\le& C \|\vv_L\|^{1/2} \|\vv_L\|^{3/2}_V \|\ovz\|_{\lL^4(\S)}
             \nonumber\\
            &\le & \frac{C}{\nu^3} \|\vv_L\|^2 \|\ovz\|^4_{\lL^4(\S)}
                  + \frac{\nu}{4} \|\vv_L\|^2_V \label{est1}
\eea
We also have
\be\lb{est2}
 \inprod{F(t)}{\vv_L} \le
 \|F(t)\|_{V^\prime} \|\vv_L\|_V \le \frac{1}{\nu} \|F(t)\|^2_{V^\prime} +
 \frac{\nu}{4} \|\vv_L\|^2_{V}.
\ee
Therefore, using Lemma \ref{lem:uprime} with the triple $V_L,H_L,V_L'$ we have
on $[0,T)$
\[
\frac{1}{2} \partial_t \|\vv_L(t)\|^2 =
- \nu\|\vv_L\|_V^2- b(\vv_L(t),\vv_L(t),\ovz(t)) +
  \inprod{F(t)}{\vv_L(t)} , \quad t\in [0,\infty).
\]
Using \eqref{est1}-\eqref{est2} we conclude that
\be\lb{equ:odevL}
 \partial_t \|\vv_L(t)\|^2 + \nu \|\vv_L(t)\|_V^2
 \le \frac{C}{\nu^3} \|\vv_L\|^2 \|\ovz\|^4_{\lL^4(\S)} + \frac{2}{\nu} \|F(t)\|^2_{V^\prime},
 \quad t\in[0,\infty).
\ee
Next by using the Gronwall lemma, we obtain
\beas
\|\vv_L(t)\|^2 &\le & \|\vv_L(0)\|^2
 \exp\left( \frac{C}{\nu^3}\int_0^t \|\ovz(\tau)\|^4_{\lL^4(\S)} d\tau \right) \\
&&+ \int_0^t \frac{2}{\nu}\|F(s)\|^2_{V^\prime}
\exp\left(\frac{C}{\nu^3}\int_{s}^t \|\ovz(\tau)\|^4_{\lL^4(\S)} d\tau\right) ds,
t \in [0,\infty).
\eeas

Let us fix $T>0$. Denoting
\[
 \Psi_T(\ovz) = \exp\left( \frac{C}{\nu^3}\int_0^T \|\ovz(\tau)\|^4_{\lL^4(\S)} d\tau \right)
  < \infty, \quad
 C_F = \int_0^T \frac{2}{\nu}\|F(s)\|^2_{V^\prime},
\]
we find that
\[
  \|\vv_L(t)\|^2 \le \|\vv_L(0)\|^2 \Psi_T(\ovz) + C_F
                 \le \|\vv(0)\|^2 \Psi_T(\ovz) + C_F < \infty,
                 \quad  t \in [0,T).
\]
Therefore,
\be\label{equ:supT}
 \sup_{t \in [0,T)} \|\vv_L(t)\|^2 \le \Psi_T(\ovz)\|\vv(0)\|^2 + C_F,
\ee
which implies that
\be\lb{concl:bounded}
 \mbox{the sequence } \{\vv_L\} \mbox{ is bounded in } L^\infty(0,T;H).
\ee
Next we integrate equation \eqref{equ:odevL} from $0$ to $T$ and
then using \eqref{equ:supT} we obtain
\beas
  \|\vv_L(T)\|^2 &+& \nu\int_0^T\|\vv_L(t)\|_V^2 dt \\
 &\le& \frac{C}{\nu^3} \int_0^T \|\ovz(t)\|^4_{\lL^4(\S)} \|\vv_L(t)\|^2 dt
       + \frac{2}{\nu} \int_0^T \|F(t)\|^2_{V^\prime} dt + \|\vv_L(0)\|^2 \\
   &\le& \frac{C}{\nu^3} (\Psi_T(\ovz)\|\vv_L(0)\|^2+C_F) \int_0^T \|\ovz(t)\|^4_{\lL^4(\S)} dt
      + \frac{2}{\nu} \int_0^T \|F(t)\|^2_{V^\prime} dt + \|\vv(0)\|^2
\eeas
The last inequality implies that
\be\lb{concl:L2bounded}
\mbox{the sequence } \{\vv_L\} \mbox{ is bounded in } L^2(0,T;V).
\ee
The claims \eqref{concl:bounded} and \eqref{concl:L2bounded} are sufficient to infer
that the sequence $\{\vv_L\}$ has a  subsequence that  converges weakly in $L^2(0,T;V)$ and weakly$^*$ in $L^\infty(0,T;H)$.
However, in order to show
that the limit function $\vv$ is a solution to our problem, we need to show that $\vv_L$
converges to $\vv$ in the strong topology of $L^2(0,T;H)$.

Let $\tvv_L = 1_{(0,T)}\vv_L$ and let the Fourier transform (in the time variable)
of $\tvv_L$ is denoted by $\hat{\vv}_L$. We want to show that
\[
 \int_{\R} |\tau|^{2\gamma} \| \hat{\vv}_L(\tau) \|^2 dt < \infty \mbox{ for some } \gamma > 0.
\]

We observe that \eqref{equ:Galerkin} can be written as
\bea
\frac{d}{dt} \tvv_L &=& \tvf_L +
       \vv_{L}(0) \delta_0 - \vv_L(T)\delta_T,\label{equ:ode0T}
       \eea

\dela{\bea
\frac{d}{dt} (\tvv_L,\Z_{\ell,m}) &=& \inprod{\tvf_L}{ \Z_{\ell,m}} +
       (\vv_{L}(0), \Z_{\ell,m}) \delta_0 - (\vv_L(T), \Z_{\ell,m})\delta_T,\label{equ:ode0T} \\
       && \quad \ell=1,\ldots,L; |m| \le L, \nonumber
\eea}
where $\delta_0,\delta_T$ are Dirac distributions at $0$ and $T$ and
\beas
  \vf_L &=& F - \nu \A \vv_L - \B\vv_L - \B(\vv_L,\ovz) -\B(\ovz,\vv_L) - \CC \vv_L \\
  \tvf_L &=& 1_{[0,T]} \vf_L
\eeas
\noindent

Applying the Fourier transform (with respect to the time variable $t$) to \eqref{equ:ode0T}
we obtain
\[
 2 i \pi \tau \hat{\vv}_L(\tau) = \hat{\vf}_L (\tau)+
  \vv_L(0) - \vv_L(T) \exp(-2i \pi T \tau), \;\;\; \tau \in  \mathbb{R},
\]
where $\hat{\vv}_L$ and $\hat{\vf}_L$ are the Fourier transforms of $\tvu_L$ and
$\tvf_L$ respectively. By multiplying this equation with
the Fourier transform of $\vv_L$ we obtain
\be\lb{equ:FT vv}
  2 i \pi \tau \| \hat{\vv}_L (\tau) \|^2 =
  \langle\hat{\vf}_L(\tau) , \hat{\vv}_L(\tau)\rangle
  + (\vv_L(0),\hat{\vv}_L(\tau)) - (\vv_L(T),\hat{\vv}_L(\tau))\exp(-2i \pi T\tau).
\ee
From the Parseval equality and \eqref{skew} and \eqref{Cuu}, we have
\[
\langle \hat{\vf}_L(\tau) , \hat{\vv}_L(\tau) \rangle
= \langle \vf_L, \vv_L \rangle
= (F,\vv_L) - \nu (\A\vv_L, \vv_L) - b(\vv_L,\ovz,\vv_L)
\]
Therefore, by using the Cauchy-Schwarz inequality and \eqref{est1}, we have
\be
|(\vf_L,\vv_L)|
\le \|F\|_{V^\prime} \|\vv_L\|_V
+ \nu \|\vv_L\|^2_V + \frac{C}{\nu^3} \|\vv_L\|^2 \|\ovz\|^4_{\lL^4(\S)}
        + \frac{\nu}{4} \|\vv_L\|^2_V
\ee
Taking into account the fact that  $\|\vv_L\|$ is bounded  (cf.\eqref{equ:supT}),
we conclude
\[
\int_0^{T} \|\vf_L\|_{V'} dt
\le \int_0^T \left(\|F\|_{V^\prime} + \frac{3\nu}{4} \|\vv_L\|_{V}
            + \frac{C_1}{\nu^3} \|\ovz\|^4_{\lL^4(\S)} \right) dt,
\]
and this remains bounded since $F \in L^2(0,T;V')$,
$\ovz \in L^4_{\loc}([0,\infty);\lL^4(\S))$ and $\vv_L$ remains
in a bounded set of $L^2(0,T;V)$ (cf.\eqref{equ:supT}).
Therefore,
\be\label{equ:boundfL}
 \sup_{\tau \in \R} \|\hat{\vf}_L(\tau)\|_{V^\prime} \le C, \quad\forall L.
\ee

Due to \eqref{equ:supT}, $\|\vv_L(0)\| < C$, $\|\vv_L(T)\| < C$, and we deduce
from \eqref{equ:FT vv} and \eqref{equ:boundfL} that
\be\label{equ:est tau}
|\tau| \| \hat{\vv}_L \|^2
\le c_2 \|\hat{\vv}_L\|_{V} + c_3 \|\hat{\vv}_L\| \le c_4 \| \hat{\vv}_L \|_{V}.
\ee



For $\gamma$ fixed, $\gamma < 1/4$, we observe that
$|\tau|^{2\gamma} \le C(\gamma) (1+|\gamma|)/(1+|\tau|^{1-2\gamma})$ for all $\tau \in \R$
and hence
\begin{eqnarray}
\int_{\R} |\tau|^{2\gamma} \|\hat{\vv}_L(\tau)\|^2 d\tau
&\le& C(\gamma) \int_{\R} \frac{1 + |\tau|} {1+ |\tau|^{1-2\gamma} }
 \|\hat{\vv}_L(\tau)\|^2 d\tau \nonumber \\
&\le& c_5 \int_{\R} \frac{|\tau|\|\hat{\vv}_L(\tau)\|^2 d\tau }{1+ |\tau|^{1-2\gamma}}
     + c_6 \int_{\R} \|\hat{\vv}_L(\tau)\|^2_V d\tau  \label{equ:taugamma}
\end{eqnarray}
The last integral of \eqref{equ:taugamma} is bounded by Parseval identity
and by \eqref{concl:L2bounded}. We bound the first integral by
\eqref{equ:est tau},
\bea
\int_{\R} \frac{|\tau|\|\hat{\vv}_L(\tau)\|^2 d\tau }{1+ |\tau|^{1-2\gamma}}
 &\le & c_4  \int_{\R}
      \frac{ \| \hat{\vv}_L(\tau) \|_V d\tau } {1 + |\tau|^{1-2\gamma}} \nonumber \\
&\le& c_4 \left( \int_{\R} \frac{d\tau}{(1+|\tau|^{1-2\gamma})^2}\right)^{1/2}
      \left( \int_0^T \|\vv_L(t)\|^2_{V} dt \right)^{1/2}
\label{equ:tau hat}
\eea
where we have used the Parseval equality and the Cauchy-Schwarz inequality
in the last step. The first integral in \eqref{equ:tau hat} are finite since
$\gamma < 1/4$ and the second integral is bounded
as $L \rightarrow \infty$ by \eqref{concl:L2bounded}.

Hence, we have shown that
\be\lb{concl:Hgamma}
 \tvv_L \mbox{ belongs to a bounded set of } \calH^{\gamma,2}(\R;V,H)
\ee
and this will enable us to apply the compactness result of Theorem~\ref{thm:compact}.

Because of \eqref{concl:bounded} and \eqref{concl:L2bounded}, without loss of
generality we may assume that there exists
$\vv \in L^2(0,T;V) \cap L^\infty(0,T;H)$ such that
\be\lb{concl:weak conv}
\begin{cases}
 \vv_L & \rightarrow \vv,  \mbox{  weakly in } L^2(0,T;V),\\
 \vv_L & \rightarrow \vv,  \mbox{  weakly-star in } L^\infty(0,T; H).
\end{cases}
\ee

Since the sphere $\S$ is bounded, the embedding $\hH^1(\S) \hookrightarrow \lL^2(\S)$ is compact,
and since \eqref{concl:Hgamma}
\be\lb{concl:bounded gamma}
 \tvv_L \mbox{ is bounded in } \calH^\gamma(0,T;\hH^1(\S), \lL^2(\S)).
\ee
By Theorem~\ref{thm:compact} the imbedding
$\calH^\gamma(0,T;\hH^1(\S), \lL^2(\S)) \hookrightarrow L^2(0,T;\lL^2(\S))$ is compact, we may
deduce from \eqref{concl:bounded gamma} that we can find a subsequence $\{\vv_{L}\}$
(which is denoted as the whole sequence for sake of simplicity of notation)
such that
\be\lb{equ:vLtov}
 \vv_L \rightarrow \vv \quad \mbox{ strongly in } L^2(0,T;\lL^2(\S)).
\ee

It remains to show that $\vv \in C([0,T];H)$ and that $\vv$ is a solution to problem
\eqref{equ:ode}. To prove the latter we take a continuously differentiable
function $\psi:[0,T] \rightarrow \R$ such that $\psi(T) = 0$.
Then by taking the inner product of \eqref{equ:Galerkin} with
$\psi(t) \vphi$ where $\vphi \in H_\ell$ for some $\ell \in \N^{+}$ then
integrating by parts, we get
\bea
&& -\int_0^T (\vv_L(t),\psi'(t)\vphi) dt = - \nu \int_0^T (P_L\A\vv_L(t),\psi(t)\vphi) dt  \nonumber\\
&& + \int_0^T (P_L \B(\vv_L(t)), \psi(t)\vphi) dt + \int_0^T (P_L \B(\vv_L(t),\ovz), \psi(t)\vphi) dt \label{intpsi}\\
&& + \int_0^T (P_L \B(\ovz,\vv_L(t)), \psi(t)\vphi) dt + \int_0^T \inprod{P_L F(t)}{ \psi(t)\vphi} dt
   + (\vv_L(0),\psi(0)\vphi). \nonumber
\eea

We aim to take the limit of \eqref{intpsi} when $L \rightarrow \infty$.
Since $\psi(\cdot)\vphi \in L^2(0,T;\lL^2(\S))$ and \eqref{equ:vLtov}
we can use the Cauchy-Schwartz inequality to infer that
\[
  \int_0^T (\vv_L(t) - \vv(t), \psi'(t)\vphi) \rightarrow 0.
\]
Therefore, the left-hand side of \eqref{intpsi} converges to
$-\int_0^T (\vv(t),\psi'(t)\vphi) dt$.

Next let us take $\ell \le L$ so that $H_\ell \subset H_L$ and $P_L\vphi = \vphi$.
For the first term on the right-hand side of \eqref{intpsi}, we observe that
\beas
 \int_0^T (P_L \A \vv_L(t),\psi(t)\vphi) dt
 &=&\int_0^T (\A \vv_L(t),\psi(t) P_L\vphi) dt \\
 &=&\int_0^T (\A \vv_L(t),\psi(t) \vphi) dt
 = \int_0^T (\vv_L(t),\psi(t) \vphi)_V dt
\eeas
Since $\psi(\cdot)\vphi \in L^2(0,T;V)$, it follows from \eqref{concl:weak conv}
that, as $L \rightarrow \infty$,
\[
 \int_0^T (P_L \vv_L(t), \psi(t)\vphi)dt -
 \int_0^T (\vv(t),\psi(t)\vphi)_V = \int_0^T (\vv_L(t)-\vv(t),\psi(t)\vphi)_V dt
  \rightarrow 0.
\]
For the second term in the right-hand side of \eqref{intpsi}, we need to prove the
following lemma
\begin{lemma}\label{lem:bvmvm}
Suppose $\vu:[0,T] \times \S \rightarrow \R^2$ is a $C^1$ function and all
first derivatives of components of $\vu$ are bounded on $\S \times [0,T]$.
Suppose
$\vv_m \rightarrow \vv$ weakly in $L^2(0,T;V)$ and strongly in
$L^2(0,T;\lL^2(\S))$. Then
\[
\int_0^T b(\vv_m(t),\vv_m(t),\vu(t))dt
\rightarrow \int_0^T b(\vv(t),\vv(t),\vu(t)) dt
\]
\end{lemma}
\begin{proof}
By \eqref{skew} we have
$
b(\vv_m,\vv_m,\vu) = -b(\vv_m,\vu,\vv_m).
$
We also have
\[
 b(\vv_m,\vu,\vv_m) - b(\vv,\vu,\vv) =
 b(\vv_m,\vu,\vv_m-\vv) + b(\vv_m-\vv,\vu,\vv).
\]
Using \eqref{b_estimate1}, we have
\[
|b(\vv_m,\vu,\vv_m-\vv)| \le C \|\vv_m\| \|\vv_m-\vv\| ( \|\Curln\vu\|_{\lL^\infty(\S)} + \|\vu\|_{\lL^\infty(\S)}).
\]
In view of \eqref{concl:L2bounded}), and $\vv_m \rightarrow \vv$ strongly
in $L^2(0,T;H)$ we conclude that
\[
  \int_0^T b(\vv_m,\vu,\vv_m-\vv) dt \rightarrow 0.
\]
Similarly, we can show that
\[
\int_0^T b(\vv_m-\vv,\vu,\vv) dt \rightarrow 0.
\]
Hence the lemma is proved.
\end{proof}
\begin{corollary}\label{cor:bvmvmL4}
Suppose $\{\vv_m\}$ is bounded in $L^\infty(0,T;H)$, $\vv \in L^\infty(0,T;H)$, and $\vv_m \rightarrow \vv$
weakly in $L^2(0,T;V)$ and strongly in $L^2(0,T;\lL^2_{\loc}(\S))$. Then for any $\vw \in L^4(0,T;\lL^4(\S))$,
\[
 \int_0^T b(\vv_m(t),\vv_m(t),\vw(t)) dt \rightarrow \int_0^T b(\vv(t),\vv(t),\vw(t)) dt.
\]
\end{corollary}


By applying Lemma~\ref{lem:bvmvm} to the second term on the right-hand side
of \eqref{intpsi} with $u(t,\xh) = \psi(t)\vphi(\xh)$, for $t\in [0,T]$,
$\xh \in \S$, and noting that
$(P_L \B(\vv_L),\psi(t)\vphi)=(\B(\vv_L),\psi P_L \vphi)=
(\B(\vv_L),\psi \vphi) = b(\vv_L,\vv_L,\psi\vphi)$, we obtain the following
convergence:
\beas
\int_0^T (P_L\B(\vv_L(t)),\psi(t)\vphi) dt
&=& \int_0^T b(\vv_L(t),\vv_L(t),\psi(t)\vphi) dt \\
&\rightarrow& \int_0^T b(\vv(t),\vv(t),\psi(t)\vphi) dt.
\eeas

We now consider the third term on the right-hand side of \eqref{intpsi}. Since
\beas
\int_0^T (P_L \B(\vv_L,\ovz),\psi(t)\vphi) dt
 &=& \int_0^T (\B(\vv_L,\ovz),\psi(t) P_L \vphi) dt\\
 &=& \int_0^T (\B(\vv_L,\ovz),\psi(t) \vphi) dt = \int_0^T b(\vv_L,\ovz,\psi(t)\vphi) dt,
\eeas
by using \eqref{skew} and \eqref{b_estimate1} we obtain
\beas
&& \left| \int_0^T (P_L \B(\vv_L,\ovz),\psi(t)\vphi) dt - \int_0^T b(\vv,\ovz,\psi(t)\vphi) dt  \right|\\
 && =  \left| \int_0^T b(\vv_L(t)-\vv(t),\ovz(t),\psi(t)\vphi) dt  \right|
  =  \left|\int_0^T b(\vv_L(t)-\vv(t),\psi(t)\vphi,\ovz) dt \right| \\
 &&\le \int_0^T | b(\vv_L(t)-\vv(t),\psi(t)\vphi,\ovz) | dt \\
 && \le C \int_0^T \| \vv_L(t) - \vv(t)\| \|\ovz\|
 ( \|\psi(t)\Curln \vphi\|_{\lL^\infty(\S)}+ \|\psi(t) \vphi \|_{\lL^\infty(\S)}).
\eeas
Since $\vv_L \rightarrow \vv$ strongly in $L^2(0,T;H)$, and
$\ovz \in L^4([0,T];\lL^4(\S))$ we conclude that the last integral converges
to $0$ as $L \rightarrow \infty$. Therefore,
\[
\int_0^T (P_L \B(\vv_L,\ovz),\psi(t)\vphi) dt
 - \int_0^T b(\vv,\ovz,\psi(t)\vphi) dt \rightarrow 0.
\]

Similarly, we have
\[
\int_0^T (P_L \B(\ovz,\vv_L),\psi(t)\vphi) dt -
\int_0^T b(\ovz,\vv(t),\psi(t)\vphi)dt \rightarrow 0.
\]

As for the fifth term on the right-hand side of \eqref{intpsi} we have
\[
 \int_0^T \inprod{P_L F}{\psi(t)\vphi} dt = \int_0^T \inprod{F}{\psi(t)\vphi} dt.
\]

Hence, by taking $L \rightarrow \infty$ in \eqref{intpsi}, we arrive at
\bea
&& -\int_0^T (\vv(t),\psi'(t)\vphi) dt = - \nu \int_0^T (\A\vv(t),\psi(t)\vphi) dt  \nonumber\\
&& + \int_0^T (\B(\vv(t)), \psi(t)\vphi) dt + \int_0^T (\B(\vv(t),\ovz), \psi(t)\vphi) dt \label{limintpsi}\\
&& + \int_0^T (\B(\ovz,\vv(t)), \psi(t)\vphi) dt + \int_0^T \inprod{F(t)}{ \psi(t)\vphi} dt
   + (\vv_0,\psi(0)\vphi). \nonumber
\eea

Since \eqref{limintpsi} has been proved for any
$\vphi \in \bigcup_{n=1}^\infty H_n$ and the set $\bigcup_{n=1}^\infty H_n$
is dense in $V$, by using a standard continuity argument we can show that
\eqref{limintpsi} holds for any $\vphi \in V$ and any $\psi \in C^1_0([0,T))$. In particular, it is
satisfied for all $\psi \in C^1_0(0,T)$. Hence,
$\vv$ solves problem \eqref{def:soln} and hence it satisfies equation
\eqref{equ:ode}.

Now we will show that $\vv \in C([0,T],H)$. Since $\vv$ solves \eqref{equ:ode},
$\vv \in L^2(0,T;V)$ and $\A:V \rightarrow V^\prime$ is a bounded linear operator,
$\A\vv \in L^2(0,T;V^\prime)$.
Since $\ovz \in L^4_{\loc}([0,\infty);\lL^4(\S)) \cap L^2_{\loc}([0,\infty);V^\prime)$, it follows
from Lemma~\ref{lem:bL4} that all terms $-\B(\ovz)+\alpha \ovz + \vf \in L^2(0,T;V^\prime)$,
$\B(\vv)$,$\B(\vv,\ovz)$,$\B(\ovz,\vv)$ belong to $L^2(0,T;V^\prime)$.
Hence $\partial_t \vv \in L^2(0,T;V^\prime)$. Thus, it follows from Lemma~\ref{lem:uprime} that
$\vv \in C([0,T];H)$.

Next, we will show that  $\vv(0) = \vv_0$. Recall
that $\vv \in L^2(0,T;V) \cap C([0,T];H)$. $\partial_t \vv \in L^2(0,T;V^\prime)$ and
$\vv$ satisfies \eqref{equ:ode}. Let us take an arbitrary function $\vphi \in V$ and
$\psi \in C_0^1([0,T))$ such that $\psi(0)=1$. Multiplying equation \eqref{equ:ode} by
$\psi(t)\vphi$ and then using integration by parts, we obtain

\bea
&& -\int_0^T (\vv(t),\psi'(t)\vphi) dt = - \nu \int_0^T (\A\vv(t),\psi(t)\vphi) dt  \nonumber\\
&& + \int_0^T (\B(\vv(t)), \psi(t)\vphi) dt + \int_0^T (\B(\vv(t),\ovz), \psi(t)\vphi) dt
 \label{limintpsi v0}\\
&& + \int_0^T (\B(\ovz,\vv(t)), \psi(t)\vphi) dt + \int_0^T \inprod{F(t)}{ \psi(t)\vphi} dt
   + (\vv(0),\psi(0)\vphi). \nonumber
\eea

By comparing equation \eqref{limintpsi} to \eqref{limintpsi v0} we infer that
$(\vv_0 - \vv(0),\vphi)\psi(0) = 0$. Since $\psi(0)=1$ we infer that $(\vv_0 - \vv(0),\vphi)=0$,
for all $\vphi \in V$. Hence, since $V$ is dense in $H$, we obtain $\vv(0) = \vv_0$.

\vspace{0.5cm}

\noindent
{\bf Part II. Uniqueness of solutions}
This is based on the proof of the uniqueness of solutions due to Lions-Prodi \cite{LioPro59};
see also Theorem III.3.2 in \cite{Tem79}. Let us assume that $\vv_1$ and $\vv_2$ are two solutions
of \eqref{equ:ode}, and we let $\vw = \vv_1 - \vv_2$. Then, by definition both
$\vv_1$ and $\vv_2$ (and hence $\vw$ as well) belong to $L^2(0,T;V) \cap C([0,T],H)$, and
by the argument above their weak time derivatives belong to $L^2(0,T;V^\prime)$. Moreover, $\vw$ solves
the following:
\be\lb{equ:vecw}
\begin{cases}
  \partial_t \vw  + \nu \A \vw &= -\B(\vw,\ovz)-\B(\ovz,\vw) - \B(\vw,\vv_1) - \B(\vv_2,\vw),  \\
   \vw(0) &= 0.
\end{cases}
\ee
The regularity of $\vw$ allows us to integrate \eqref{equ:vecw} against
$\vw$ and then use Lemma~\ref{lem:uprime} and \eqref{skew} to obtain
\[
\partial_t \|\vw\|^2 + 2\nu \|\vw\|_V^2 = -2b(\vw,\ovz,\vw)-2b(\vw,\vv_1,\vw).
\]
By using \eqref{skew},\eqref{b_estimate3},\eqref{equ:L4}, and then the Young
inequality, we get
\beas
 \partial_t \|\vw\|^2 + 2 \nu \|\vw\|^2_V &\le& C\|\vw\|^{1/2}\|\vw\|^{3/2}_V
      (\|\ovz\|_{\lL^4(\S)} + \|\vv_1\|_{\lL^4(\S)}) \\
     &\le& \frac{3 \nu} {4} \|\vw\|^2_V +
      \frac{C}{\nu^3} \|\vw\|^2(\|\vv_1\|^4_{\lL^4(\S)} + \|\ovz\|^4_{\lL^4(\S)}).
\eeas
Therefore,
\[
 \partial_t \|\vw(t)\|^2 \le \frac{C}{\nu^3} \|\vw\|^2(\|\vv_1\|^4_{\lL^4(\S)} + \|\ovz\|^4_{\lL^4(\S)}) \mbox{ a.e. on } (0,T).
\]
Since $\int_0^T \|\vv_1\|^4_{\lL^4(\S)} + \|\ovz\|^4_{\lL^4(\S)} dt < \infty$ and
$\vw(0) = 0$, by applying the Gronwall Lemma, we infer that $\|\vw(0)\|^2=0$ for all $t \in [0,T]$.
This means that $\vv_1(t) =  \vv_2(t)$ for all $t\in [0,T]$, which proves the
uniqueness of the solution.

\subsection{Proof of Theorem~\ref{thm:limit}}
We introduce the following notations:
\begin{align*}
\vv_n(t) &= \vv(t,\ovz_n), \quad \vv(t) = \vv(t,z), \quad \vy_n(t) = \vv(t,\ovz_n) - \vv(t,\ovz), \quad t \in [0,T], \\
\what{\ovz}_n &= \ovz_n - \ovz, \quad \what{\vf}_n = \vf_n - \vf.
\end{align*}
It is easy to see that $\vy_n(t)$ solves the following initial value problem:
\begin{equation}
\begin{cases}
 \partial_t \vy_n(t) &= \nu\A \vy_n(t) - \B(\vv_n(t)+ \ovz_n(t)) + \B(\vv(t)+\ovz(t)) - \CC \vy_n +
           \alpha \what{\ovz}_n + \what{\vf}_n, \\
 \vy_n(0) &= \vu_{0n} - \vu_0.
\end{cases}
\end{equation}
It follows from Lemma~\ref{lem:uprime} that $\frac{1}{2} \partial_t \|\vy_n(t)\|^2 = (\partial_t \vy_n(t),\vy_n(t))$.
Therefore, noting that $(\CC \vy_n,\vy_n)=0$,
\begin{align*}
\frac{1}{2} \partial_t \|\vy_n(t)\|^2 +& \nu (\A \vy_n,\vy_n) = -b(\vy_n,\vv_n,\vy_n)-b(\vv,\vy_n,\vy_n) \\
      -& b(\what{\ovz}_n,\vv_n,\vy_n) - b(\ovz,\vy_n,\vy_n) - b(\vv_n,\what{\ovz}_n,\vy_n) \\
      -& b(\vy_n,\ovz,\vy_n) - b(\ovz_n,\what{\ovz}_{n},\vy_n) - b(\what{\ovz}_n,\ovz,\vy_n) \\
      +& \alpha(\what{\ovz}_n,\vy_n) + (\what{\vf}_n,\vy_n), \quad t\ge 0.
\end{align*}
By using the Young inequality, in view of inequalities \eqref{equ:L4} and \eqref{b_estimate3}, we infer that
\begin{align*}
b(\vy_n,\vv_n,\vy_n) &\le \|\vy_n\|^2_{\lL^4(\S)} \|\vv_n\|_V
                      \le \|\vy_n\| \|\vy_n\|_V \|\vv_n\|_V  \\
                     &\le \frac{\nu}{20} \|\vy_n\|^2_V + \frac{5}{\nu} \|\vv_n\|_V^2 \|\vy_n\|^2, \\
b(\vv,\vy_n,\vy_n) & \le \|\vv\|_{\lL^4(\S)} \|\vy_n\|_V \|\vy_n\|_{\lL^4(\S)}
                      \le \|\vv\|_{\lL^4(\S)} \|\vy_n\|^{3/2}_V  \|\vy_n\|^{1/2}  \\
                   & \le \frac{\nu}{20} \|\vy_n\|^2_V +
                    \frac{15^3}{4\nu^3} \|\vy_n\|^2 \|\vv\|^4_{\lL^4(\S)} \\
b(\what{\ovz}_n, \vv_n,\vy_n)
             & \le \|\what{\ovz}_n \|_{\lL^4(\S)} \|\vy_n\|_V \|\vv_n\|_{\lL^4(\S)}    \\
             & \le \frac{\nu}{20} \|\vy_n\|^2_V +
                  \frac{5}{\nu} \|\what{\ovz}_n\|^2_{\lL^4(\S)} \|\vv_n\|\|\vv_n\|_V \\
b(\ovz,\vy_n,\vy_n) &\le \|\ovz\|_{\lL^4(\S)} \|\vy_n\|_V \|\vy_n\|_{\lL^4(\S)}
                     \le \|\ovz\|_{\lL^4(\S)} \|\vy_n\|^{3/2}_V  \|\vy_n\|^{1/2}  \\
                    &\le \frac{\nu}{20} \|\vy_n\|^2_V +
                         \frac{15^3}{4\nu^3} \|\vy_n\|^2 \|\ovz\|^4_{\lL^4(\S)}
\end{align*}
\begin{align*}
b(\vv_n,\what{\ovz}_n,\vy_n) & \le \|\vv_n\|_{\lL^4(\S)} \|\vy_n\|_V \|\what{\ovz}_n\|_{\lL^4(\S)} \\
                  & \le\frac{\nu}{20} \|\vy_n\|^2_V +
                   \frac{5}{\nu} \|\vv_n\| \|\vv_n\|_V \|\what{\ovz}_n\|^2_{\lL^4(\S)} \\
b(\vy_n,\ovz,\vy_n) &  \le\|\vy_n\|^2_{\lL^4(S)} \|\ovz\|_V
                       \le \|\vy_n\| \|\vy_n\|_V \|\ovz\|_V  \\
                    & \le \frac{\nu}{20} \|\vy_n\|^2_V + \frac{5}{\nu} \|\ovz\|_V^2 \|\vy_n\|^2,  \\
b(\ovz_n,\what{\ovz}_n,\vy_n) & \le \|\ovz_n\|_{\lL^4(\S)} \|\vy_n\|_V \|\what{\ovz}_n\|_{\lL^4(\S)} \\
                           & \le \frac{\nu}{20} \|\vy_n\|^2_V +
                                 \frac{5}{\nu} \|\ovz_n\|^2_{\lL^4(\S)} \|\what{\ovz}_n\|^2_{\lL^4(\S)} \\
b(\what{\ovz}_n,\ovz,\vy_n) & \le
    \|\what{\ovz}_n\|_{\lL^4(\S)} \|\vy_n\|_V \|\ovz\|_{\lL^4(\S)} \\
      & \le \frac{\nu}{20} \|\vy_n\|^2_V +
        \frac{5}{\nu} \|\ovz\|^2_{\lL^4(\S)} \|\what{\ovz}_n\|^2_{\lL^4(\S)} \\
\alpha(\what{\ovz}_n,\vy_n) &\le \alpha \|\vy_n\|_V \|\what{\ovz}_n\|_{V^\prime}      \\
           & \le \frac{\nu}{20} \|\vy_n\|^2_V
             + \frac{5 \alpha^2}{\nu}\|\what{\ovz}_n\|^2_{V^\prime}, \\
(\what{\vf}_n,\vy_n) &\le \|\vy_n\|_V \|\vf_n\|_{V^\prime} \\
        &\le \frac{\nu}{20} \|\vy_n\|^2_V
             + \frac{5}{\nu}\|\what{\vf}_n\|^2_{V^\prime}.
\end{align*}
Hence we have, weakly on $(0,T)$,
\begin{align*}
\partial_t \|\vy_n\|^2 + \nu \|\vy_n\|^2_V
   &\le \frac{10}{\nu} \|\vv_n\|_V^2 \|\vy_n\|^2
     + \frac{15^3}{2\nu^3} \|\vy_n\|^2 \|\vv\|^4_{\lL^4(\S)} \\
   & +\frac{10}{\nu} \|\what{\ovz}_n\|^2_{\lL^4(\S)} \|\vv_n\|\|\vv_n\|_V
     +\frac{15^3}{2\nu^3} \|\vy_n\|^2 \|\ovz\|^4_{\lL^4(\S)} \\
   & + \frac{10}{\nu} \|\vv_n\| \|\vv_n\|_V \|\what{\ovz}_n\|^2_{\lL^4(\S)}
     + \frac{10}{\nu} \|\ovz\|_V^2 \|\vy_n\|^2 \\
   & +  \frac{10}{\nu} \|\ovz_n\|^2_{\lL^4(\S)} \|\what{\ovz}_n\|^2_{\lL^4(\S)}
     +\frac{10}{\nu} \|\ovz\|^2_{\lL^4(\S)} \|\what{\ovz}_n\|^2_{\lL^4(\S)} \\
   & + \frac{10 \alpha^2}{\nu}\|\what{\ovz}_n\|^2_{V^\prime}
     + \frac{10}{\nu}\|\what{\vf}_n\|^2_{V^\prime}.
\end{align*}
Integrating the above inequality from $0$ to $t$, for $ t\in [0,T]$, we get
\begin{equation}\label{equ:int ineq}
\begin{aligned}
\|\vy_n(t)\|^2 &+\nu \int_0^t \|\vy_n(s)\|_V^2 ds \le \|\vy_n(0)\|^2  \\
&+\frac{10}{\nu} \int_0^t \beta_n(s) ds  + \int_0^t \gamma_n(s)\|\vy_n(s)\|^2 ds,
\quad t \in [0,T],
\end{aligned}
\end{equation}
where
\begin{align*}
\beta_n & = \|\what{\ovz}_n\|^2_{\lL^4(\S)} \|\vv_n\|\|\vv_n\|_V
       + \|\vv_n\| \|\vv_n\|_V \|\what{\ovz}_n\|^2_{\lL^4(\S)}
       +  \|\ovz_n\|^2_{\lL^4(\S)} \|\what{\ovz}_n\|^2_{\lL^4(\S)} \\
      & +   \|\ovz\|^2_{\lL^4(\S)} \|\what{\ovz}_n\|^2_{\lL^4(\S)}
       + \alpha^2\|\what{\ovz}_n\|^2_{V^\prime}
       +\|\what{\vf}_n\|^2_{V^\prime}, \\
\gamma_n & = \frac{10}{\nu} \|\vv_n\|_V^2
          +\frac{15^3}{2\nu^3}  \|\vv\|^4_{\lL^4(\S)}
          +\frac{15^3}{2\nu^3}  \|\ovz\|^4_{\lL^4(\S)}
          + \frac{10}{\nu} \|\ovz\|_V^2.
\end{align*}
Then by the Gronwall inequality,
\[
 \|\vy_n(t)\|^2 \le \left( \|\vy_n(0)\|^2
 + \frac{10}{\nu} \int_0^t \beta_n(s) ds \right) \exp \left(\int_0^t \gamma_n(s) ds \right).
\]
We observe that
\begin{align*}
\int_0^T \beta_n(s) ds &=
\int_0^T [\|\what{\ovz}_n(s)\|^2_{\lL^4(\S)} \|\vv_n(s)\|\|\vv_n(s)\|_V
       + \|\vv_n(s)\| \|\vv_n(s)\|_V \|\what{\ovz}_n(s)\|^2_{\lL^4(\S)} \\
     & +  \|\ovz_n(s)\|^2_{\lL^4(\S)} \|\what{\ovz}_n(s)\|^2_{\lL^4(\S)}
      +   \|\ovz(s)\|^2_{\lL^4(\S)} \|\what{\ovz}_n(s)\|^2_{\lL^4(\S)} \\
     &  + \alpha^2\|\what{\ovz}_n(s)\|^2_{V^\prime}
       +\|\what{\vf}_n(s)\|^2_{V^\prime} ] ds, \\
&\le \big[ 2\|\vv_n\|_{L^\infty(0,T;H)} \|\vv_n\|_{L^2(0,T;V)} + \\
      &  \qquad  \|\ovz_n\|^2_{L^4(0,T;\lL^4)} + \|\ovz\|^2_{L^4(0,T;\lL^4)} \big]
        \|\what{\ovz}_n\|^2_{L^4(0,T;\lL^4)}    \\
     & \qquad + \alpha^2 \|\what{\ovz}_n\|^2_{L^2(0,T;V^\prime)} + \|\what{\vf}_n\|^2_{L^2(0,T;V^\prime)}.
\end{align*}

Therefore, $\int_0^T \beta_n(s) ds \rightarrow 0$, as $ n \rightarrow \infty$.
Since $\|\vy_n(0)\| \rightarrow 0$, as $n \rightarrow \infty$ and for some constant $C < \infty$
and all $n \in \mathbb{N}$,
\begin{align*}
\int_0^T \gamma_n(s) ds &= \int_0^T
  \left( \frac{10}{\nu} \|\vv_n\|_V^2
          +\frac{15^3}{2\nu^3}  \|\vv\|^4_{\lL^4(\S)}
                    +\frac{15^3}{2\nu^3}  \|\ovz\|^4_{\lL^4(\S)}
                              + \frac{10}{\nu} \|\ovz\|_V^2 \right) ds \\
 &\le \frac{10}{\nu} \|\vv_n\|^2_{L^2(0,T;V)} +
      +\frac{15^3}{2\nu^3} \|\vv\|^2_{L^2(0,T;H)} \|\vv\|^2_{L^2(0,T;V)} \\
 &   \quad + \frac{15^3}{2\nu^3}  \|\ovz\|^4_{L^4(0,T;\lL^4)}
      +\frac{10}{\nu}  \|\ovz\|^2_{L^2(0,T;V)} \\
 &\le C,
\end{align*}
we infer that $\vy_n(t) \rightarrow 0$ in $H$ as $n \rightarrow \infty$, uniformly in $ t\in [0,T]$.
In other words,
\[
  \vv(\cdot,\ovz_n) \vu_{0n} \rightarrow \vv(\cdot,\ovz) \vu_0 \quad \mbox{ in } C([0,T];H).
\]
From inequality \eqref{equ:int ineq} we also have
\begin{align*}
\nu \int_0^T \|\vy_n(s)\|^2_V ds &\le \|\vy_n(0)\|^2 +
\frac{10}{\nu} \int_0^T \beta_n(s) ds
+ \int_0^T \gamma_n(s) \|\vy_n(s)\|^2 ds \\
&\le \|\vy_n(0)\|^2 + \frac{10}{\nu} \int_0^T \beta_n(s) ds +
    \sup_{ s\in [0,T] } \|\vy_n(s)\|^2 \int_0^T \gamma_n(s) ds.
\end{align*}
Hence, $\int_0^T \|\vy_n(s)\|_V^2 ds \rightarrow 0$ as $n \rightarrow \infty$ and therefore,
\[
\vv(\cdot,\ovz_n) \vu_{0n} \rightarrow \vv(\cdot,\ovz) \vu_0 \quad \mbox{ in } L^2([0,T];V).
\]
\section{Ornstein-Uhlenbeck processes}\label{OU}
\subsection{Preliminaries}
Let $\gamma_K$ be a standard cylindrical Gaussian measure on a real separable Hilbert space $K$.
Let us recall that for a real separable Banach space $X$, a
bounded linear operator $U:K \rightarrow X$ is said to be $\gamma$-radonifying iff the
cylindrical measure $\nu_U=\gamma_K\circ U^{-1}$ extends (uniquely) to a countably
additive probability measure on the Borel $\sigma$-algebra of $X$.  By $R(K,X)$ we denote
the Banach space of $\gamma$-radonifying operators from $K$
to $X$ with the norm
\[
  \|U\|_{R(K,X)} := \left( \int_X |x|^2_X \nu_U(dx) \right)^{1/2}, \quad U \in R(K,X).
\]

Let $U:H\to H$ be a bounded symmetric operator with the complete orthonormal
system of
eigenfunctions $\left(\mathbf{e}_l\right)\subset \mathbb L^p\left(\S\right)$ and
the corresponding
set of eigenvalues $\left(\lambda_l\right)$. It follows from \cite[Theorem 2.3]{BrzNee03}
that for a self adjoint operator $U \ge cI$ in $H$, where $c>0$, such that $U^{-1}$
is compact, the operator
$U^{-s}:H \rightarrow \lL^{p}(\S)$ is well defined and $\gamma$-radonifying iff
\begin{equation}\label{equ:ref9}
\int_{\S} \left[ \sum_{\ell} \lambda^{-2s}_\ell | \ve_\ell(\x) |^2 \right]^{p/2} dS(\x) < \infty.
\end{equation}

\begin{lemma}\label{lem:radon}
Let $\DDelta$ denote the Laplace--de Rham operator on $\S$. Then the
operator
\begin{equation}
 (-\DDelta)^{-s}: H \rightarrow \lL^4(\S) \text{ is } \gamma-
 \text{radonifying iff } s>1/2.
\end{equation}
\end{lemma}
\begin{proof}
Let us recall that all the distinct eigenvalues of $-\DDelta$ are $\lambda_\ell = \ell(\ell+1)$, $\ell=0,1,\ldots$ and the corresponding eigenfunctions are
given by the divergence free vector spherical harmonics $\vY_{\ell,m}$
for $|m| \le \ell$, $\ell \in {\mathbb N}$
\cite[page 216]{VarMosKhe88}.
Let us recall also the addition theorem for vector spherical harmonics
\cite[formula (81), page 221]{VarMosKhe88}
\[
\sum_{|m|\le \ell} |\vY_{\ell,m}(\x)|^2 = \frac{2\ell+1}{4\pi} P_{\ell}(1),
\;\; \x \in \S,
\]
and the fact that $P_\ell(1)=1$ with $P_\ell$ being the Legendre
polynomial of degree $\ell$. Therefore, \eqref{equ:ref9} yields
\begin{align*}
 \int_{\S} &\left[ \sum_{\ell=0} (\ell(\ell+1))^{-2s}\sum_{|m|\le \ell} |\vY_{\ell,m}|^2 \right]^{4/2} dS \\
 &= \int_{\S} \left[ \sum_{\ell=0} (\ell(\ell+1))^{-2s} \frac{2\ell+1}{4\pi} P_\ell(1) \right]^2 dS < \infty
\end{align*}
if and only if $s>\frac{1}{2}$ and the lemma follows.
\end{proof}
Let $X=\lL^4(\S) \cap H$ denote the Banach space endowed with the norm
\[
  \|x\|_{X} = \|x\|_H + \|x\|_{\lL^4(\S)}.
\]
It can be shown that $X$ is an $M$-type 2 Banach space, see \cite{Brz97} for details.

It follows from Lemma~\ref{lem:radon} that the operator
\begin{equation}\label{remark}
 \fourIdx{}{}{\!-s}{}\A: H \rightarrow \lL^4(\S) \cap H \text{  is } \gamma \text{-radonifying if }
 s> 1/2.
\end{equation}
Let us recall, that the Stokes operator $\A$ generates an analytic $C_0$-semigroup $\{e^{-t \A}\}_{t\ge 0}$ in $X$. While this fact is known, we could not find a direct reference. We provide a brief argument in Theorem \ref{anal} in the Appendix.
Since the Coriolis operator $\CC$ is bounded on $X$ we can define in $X$ an operator
\[
\hat{\A} = \nu \A + \CC,\quad\mathrm{dom}(\hat{\A})=\mathrm{dom}(\A),
\]
with $\nu > 0$.

\begin{proposition}\label{prop-rad}
The operator $\hat{\A}$ with the domain $\mathrm{dom}(\hat{\A})=\mathrm{dom}(\A)$ generates a strongly continuous and analytic semigroup
$\{e^{-t{\hat{\A}}}\}_{t\ge 0}$ in $X$.
Then there exist certain $M\ge 1$ and $\mu>0$
\begin{equation}\label{exp}
\|e^{-t\hat{\A}}\|_{\calL(X,X)} \le M e^{-\mu t},\quad t\ge 0.
\end{equation}
Moreover, for any $\delta>0$ there exists $M_\delta \ge 1$ such that
\begin{equation}\label{exp2}
\|\fourIdx{}{}{\!\delta}{}{\hat{\A}} e^{-t\hat{\A}}\|_{\calL(X,X)} \le
M_\delta t^{-\delta} e^{-\mu t},\quad t > 0.
\end{equation}
\end{proposition}
Before we embark with the proof of the above result let us formulate the following quite obvious its consequence.
\begin{corollary}\label{cor-rad0} In the framework of Proposition \ref{prop-rad} let us additionally  assume  that
there exists a separable Hilbert space $K\subset X$ such that
$\hat{\A} X \subset K$ and the operator $\hat{\A}^{-\delta}:K\to X$
is $\gamma$-radonifying for some $\delta >0$. Then
\[
\int_0^\infty\left\|e^{-t\hat{\A}}\right\|^2_{R(K,X)}\, dt<\infty.
\]
\end{corollary}

\begin{proof}[Proof of Proposition \ref{prop-rad}]
Since $\CC$ is a bounded linear operator on $X$, Theorem \ref{anal} and Corollary 2.2 on p. 81 of \cite{Paz83} imply that the operator $A$
is a generator of another analytic $C_0$ semigroup on $X$.

Suppose $\vu(t) = e^{-t(\nu\A+\CC)} \vu_0$ for some $\vu_0 \in V\subset X$.
We will show first that
\be\lb{equ:int1}
  \int_{0}^\infty \|\vu(t)\|^2 dt =
  \int_{0}^\infty \|e^{-t (\nu\A + \CC)} \vu_0\|^2_{\lL^2(\S)} < \infty.
\ee
By Lemma~\ref{lem:uprime} and equation \eqref{Cuu} we have
\[
 \frac{1}{2} \frac{d}{dt} \|\vu(t)\|^2 = (\vu'(t),\vu(t))
      = - (\nu \A \vu,\vu) - (\CC \vu, \vu) = -\nu \|\vu\|^2_V.
\]
Hence
\be\lb{equ:ode1}
  \frac{1}{2} \frac{d}{dt} \|\vu(t)\|^2 + \nu \|\vu(t)\|^2_V = 0.
\ee
Since $\|\vu\|^2_V = \|\A^{1/2}\vu\|^2 \ge \lambda_1 \|\vu\|^2$, we obtain
\[
 \frac{1}{2} \frac{d}{dt} \|\vu(t)\|^2 \le  - 2\lambda_1 \nu \|\vu(t)\|^2.
\]
Using the Gronwall inequality, we obtain $\|\vu(t)\|^2 \le e^{-2\lambda_1 \nu t} \|\vu_0\|^2$ for $\vu_0\in\mathbb H^1\left(\S\right)$, hence for all $\vu_0\in X$ by the density argument and  \eqref{equ:int1} follows.

We also deduce from \eqref{equ:ode1} that
\[
  \|\vu(T)\|^2 + 2 \nu \int_{0}^T  \|\vu(t)\|_V^2 dt =  \|\vu(0)\|^2,
\]
and letting $T \rightarrow \infty$ we obtain
\be\lb{equ:int2}
 \int_{0}^\infty  \|\vu(t)\|_V^2  < \infty,
\ee
for all $\vu_0\in H$. Using inequality \eqref{equ:L4}, the Cauchy-Schwarz inequality and inequalities \eqref{equ:int1}--
\eqref{equ:int2} we obtain
\[
 \int_0^\infty \|\vu(t)\|_{\lL^4}^2 dt\le \left(\int_0^\infty \|\vu(t)\|^2_{\lL^2} dt\right)^{1/2}
          \left(\int_0^\infty \|\vu(t)\|^2_{V} dt\right)^{1/2}< \infty.
\]
Using Theorem 1.1 on p. 116 of \cite{Paz83} with $X = \lL^4(\S) \cap H$
we conclude that $\|e^{-tA}\|_{\calL(X,X)} \le Me^{-\mu t}$ for some constants $M \ge 1$ and $\mu >0$.

Using \cite[Theorem 6.13, page 74]{Paz83} with $X = \lL^4(\S) \cap H$ we arrive
at conclusion \eqref{exp2}.
\end{proof}
Let $E$ be the completion of $\fourIdx{}{}{\!\!-\delta}{}\A(X)$
with respect to the image norm
\[\|\vv\|_E = \|\fourIdx{}{}{\!\!-\delta}{}\A \vv\|_X,\quad \vv\in X.\]
 For
$\xi \in (0,1/2)$ we set
\[
  C^\xi_{1/2}(\R,E) := \{\omega \in C(\R,E): \omega(0) = 0,
  \sup_{t,s\in \R} \frac{|\omega(t)-\omega(s)|_E}{|t-s|^\xi(1+|t|+|s|)^{1/2}} < \infty\}.
\]
The space $C^\xi_{1/2}(\R,E)$ equipped with the the norm
\[
 \|\omega\|_{C^\xi_{1/2}(\R,E)} =
 \sup_{t \ne s\in \R} \frac{|\omega(t)-\omega(s)|_E}{|t-s|^\xi(1+|t|+|s|)^{1/2}}
\]
is a nonseparable Banach space. However, the closure of $\{\omega \in C^\infty_0(\R):
\omega(0)=0\}$ in $C^\xi_{1/2}(\R,E)$, denoted by $\Omega(\xi,E)$, is a separable Banach
space.

Let us denote by $C_{1/2}(\R,X)$ the space of all continuous functions
$\omega:\R \rightarrow X$ with $O(1+|t|^{1/2})$ growth condition, i.e. for some $C = C(\omega)>0$,
\be
 |\omega(t)| \le C(1+|t|^{1/2}), \quad t\in \R.
\ee
The space $C_{1/2}(\R,E)$ endowed with the norm
\[
 \|\omega\|_{C_{1/2}(\R,E)} = \sup_{t\in \R} \frac{|\omega(t)|_E}{1+|t|^{1/2}}
\]
is a nonseparable Banach space.

We denote by $\calF$ the Borel $\sigma$-algebra on $\Omega(\xi,E)$. One can show that \cite{Brz96}
for $\xi \in (0,1/2)$, there exists a Borel probability measure $\bbP$ on $\Omega(\xi,E)$ such that
the canonical process $w_t$, $t\in \R$, defined by
\begin{equation}\label{equ:Wiener}
w_t(\omega) := \omega(t), \quad \omega \in \Omega(\xi,E),
\end{equation}
is a two-sided Wiener process such that the Cameron-Martin (or Reproducing Kernel Hilbert) space
of the Gaussian measure $\calL(w_1)$ on E is equal to $K$. For $t \in \R$, let
$\calF_t := \sigma\{w_s: s\le t\}$. Since for each $t \in \R$ the map
$z \circ i_t: E^* \rightarrow L^2(\Omega(\xi,E),\calF_t,\bbP)$, where
$i_t: \Omega(\xi, E) \ni \gamma \mapsto \gamma(t) \in E$, satisfies
$\E|z\circ i_t|^2 = t|z|^2_K$, there exists a unique extension of $z\circ i_t$ to a bounded
linear map $W_t: K \rightarrow L^2(\Omega(\xi,E),\calF_t,\bbP)$. Moreover, the family
$(W_t)_{t\in \R}$ is an $H$-cylindrical Wiener process on a filtered probability space
$(\Omega(\xi,E), ({\mathfrak F})_{t \in \R},\bbP)$ in the sense of e.g. \cite{BrzPes01}.

\subsection{Ornstein-Uhlenbeck process}
The following is our standing assumption.
\begin{assumption}\label{ass:radon}
There exists a separable Hilbert space $K \subset H \cap \lL^4(\S)$ such that for a certain
$\delta \in (0,1/2)$, the operator
\be\lb{radonify}
\fourIdx{}{}{\!\!-\delta}{}\A : K \rightarrow H \cap \lL^4(\S) \mbox{ is }\gamma\mbox{-radonifying.}
\ee
\end{assumption}
It follows from \eqref{remark} that if $K={\mathcal D}(\A^s)$ for some $s>0$, then
Assumption~\ref{ass:radon} is satisfied. See also remark 6.1 in \cite{BrzLi06}.

On the space 
$\Omega(\xi,E)$
we consider a flow $\vartheta = (\vartheta_t)_{t\in \R}$ defined by
\[
  \vartheta_t \omega(\cdot) = \omega(\cdot + t) - \omega(t), \quad \omega\in \Omega(\xi,E), \;\;t\in \R.
\]

For $\xi \in (\delta,1/2)$ and $\tilde{\omega} \in C^\xi_{1/2}(\R,X)$ we define
\be\label{zhat}
 \hat{z}(t) = \hat{z}(\hat{\A};\tilde{\omega})(t) =
 \int_{-\infty}^t \hat{\A}^{1+\delta} e^{-(t-r)\hat{\A}} (\tilde{\omega}(t)-\tilde{\omega}(r))dr, \quad t \in \R.
\ee
By Proposition~\ref{prop-rad}, for each $\delta > 0 $ there exists $C = C(\delta)>0$ such that
\be\label{semigroup}
 \|\fourIdx{}{}{\!\delta}{}{\hat{\A}} e^{-t\hat{\A}} \|_{\calL(X,X)} \le C t^{-\delta} e^{-\mu t}, \quad t  \ge 0.
\ee
 This was an assumption in \cite[Proposition 6.2]{BrzLi06}.
Rewriting that proposition in a slightly more general form we have

\begin{proposition}\label{zhat well def}
For any $\alpha\ge 0$, the operator $-(\hat{\A}+\alpha I)$ is a generator of an analytic semigroup
$\{e^{-t(\hat{\A}+\alpha I)}\}_{t \ge 0}$ in $X$ such that
\[\|\fourIdx{}{}{\!\delta}{}{\hat{\A}} e^{-t(\hat{\A}+\alpha I)} \|_{\calL(X,X)} \le C t^{-\delta} e^{-(\mu+\alpha) t}, \quad t  \ge 0.\]
If $t \in \R$, then $\hat{z}(t)$ defined in \eqref{zhat} is a well-defined element of $X$ and
the mapping $\tilde{\omega} \mapsto \hat{z}(t)$ is continuous from $C^\xi_{1/2}(\R,X)$ to $X$.
Moreover, the map $\hat{z} : C^\xi_{1/2}(\R,X) \rightarrow C_{1/2}(\R,X)$ is well defined,
linear and bounded. In particular, there exists a constant $C < \infty$ such that for
any $\tilde{\omega} \in C^{\xi}_{1/2}(\R,X)$
\be\label{equ:zhatgrow}
  | \hat{z}(\tilde{\omega}) | \le C (1+ |t|^{1/2})\|\tilde{\omega}\|_{C^{1/2}(\R,X)}.
\ee
\end{proposition}
The following results are respectively Corollary 6.4, Theorem 6.6 and Corollary 6.8
from \cite{BrzLi06}.
\begin{corollary}
Assume that $A$ is a generator of an analytic semigroup $\{e^{-tA}\}_{t \ge 0}$ such that $A$ has bounded inverse. Then for all $-\infty < a <b<\infty$
and $t \in \R$, for $\tilde{\omega} \in C^\xi_{1/2}(\R,X)$ the map
\[
 \tilde{\omega} \mapsto (\hat{z}(\tilde{\omega})(t),\hat{z}(\tilde{\omega}))
         \in X \times L^4(a,b;X)
\]
is continuous. Moreover, the above result is valid with the space $C^\xi_{1/2}(\R,X)$
being replaced by $\Omega(\xi,X)$.
\end{corollary}

\begin{theorem}\label{thm:vartheta}
Assume that $A$ is a generator of an analytic semigroup $\{e^{-tA}\}_{t \ge 0}$ such that $A$ has bounded inverse.Then for any
$\omega \in C^{\xi}_{1/2}(\R,X)$,
\[
 \hat{z}(\vartheta_s \omega(t)) = \hat{z}(\omega)(t+s), \quad t,s \in \R.
\]
In particular, for any $\omega \in \Omega$ and all $t,s \in \R$,
$\hat{z}(\vartheta_s \omega)(0) = \hat{z}(\omega)(s)$.
\end{theorem}

For $\xi \in C_{1/2}(\R,X)$ we put
\[
  (\tau_s \zeta) = \zeta(t+s), \quad t,s \in \R.
\]
Thus, $\tau_s$ is a linear a bounded map from $C_{1/2}(\R,X)$ into itself. Moreover,
the family $(\tau_s)_{s\in \R}$ is a $C_0$ group on $C_{1/2}(\R,X)$.

Using this notation Theorem~\ref{thm:vartheta} can be rewritten in the following way.
\begin{corollary}\label{cor:flow}
For $s\in \R$, $\tau_s \circ \hat{z} = \hat{z} \circ \vartheta_s$, i.e.
\[
 \tau_s(\hat{z}(\omega)) =
 \hat{z}(\vartheta_s(\omega)), \quad \omega \in C^\xi_{1/2}(\R,X).
\]
\end{corollary}
We define
\[
\ovz_\alpha(\omega) := \hat{z}(\hat{\A} + \alpha I;
          (\hat{\A}+\alpha I)^{-\delta}\omega) \in C_{1/2}(\R,X),
\]
i.e. for any $t\ge 0$,
\begin{eqnarray}
  \ovz_\alpha(\omega)(t) &:=&
   \int_{-\infty}^t (\hat{\A}+\alpha I)^{1+\delta}
   e^{-(t-r)(\hat{ \A} + \alpha I)} \label{zalpha} \\
   && [ (\hat{ \A} +\alpha I)^{-\delta} \omega(t) -
   (\hat{ \A} +\alpha I)^{-\delta} \omega(r)   ] dr
   \nonumber
\end{eqnarray}

By the fundamental theorem of calculus, we obtain
\begin{eqnarray*}
\frac{d \ovz_\alpha(t)}{dt} &=& - (\hat{\A} + \alpha I) \int_{-\infty}^t
         (\hat{\A}+\alpha I)^{1+\delta} e^{-(t-r)
         (\hat{ \A} + \alpha I)} \\
&& [ (\hat{ \A}  +\alpha I)^{-\delta} \omega(t) -
(\hat{ \A } +\alpha I)^{-\delta} \omega(r)   ] dr
  + \dot{\omega}(t),
\end{eqnarray*}
where $\dot{\omega}(t) = d\omega(t)/dt$.
Hence $\ovz_\alpha(t)$ is the solution of the following equation
\be\lb{OUequ}
 \frac{d\ovz_\alpha(t)}{dt} + (\hat{\A} + \alpha I) \ovz_\alpha = \dot{\omega}(t), \quad t \in \R.
\ee

It follows from Theorem~\ref{thm:vartheta} that
\begin{equation}\label{equ:zalpha}
\ovz_\alpha(\vartheta_s \omega)(t) = \ovz_\alpha(\omega)(t+s),
\quad \omega \in C^{\xi}_{1/2}(\R,X), \; t,s \in \R.
\end{equation}

Similar to our definition \eqref{equ:Wiener} of the Wiener process $w(t)$, $t \in \R$,
we can view the formula \eqref{zalpha} as a definition of a process $\ovz_\alpha(t)$, $t \in \R$,
on the probability space $(\Omega(\xi,E),\calF,\bbP)$. Equation \eqref{OUequ} suggests that this
process is an Ornstein-Uhlenbeck process.

\begin{proposition}\label{pro:stationary}
The process $\ovz_\alpha(t)$, $t\in \R$, is a stationary Ornstein-Uhlenbeck process.
It is the solution of the equation
\[
 d\ovz_\alpha(t) + (\hat{ \A}  +\alpha I)\ovz_\alpha dt = dw(t), \quad t\in \R,
\]
i.e. for all $t\in \R$, a.s.
\begin{equation}\label{equ:integral}
  \ovz_\alpha(t) = \int_{-\infty}^t e^{-(t-s)(\hat{\A} + \alpha I)} dw(s),
\end{equation}
where the integral is the It\^{o} integral on the $M$-type 2 Banach space $X$ in
the sense of \cite{Brz97}.

In particular, for some constant $C$ depending on $X$,
\begin{align}
 \E|\ovz_\alpha(t)|^2_X &=
  \E \left| \int_{-\infty}^t e^{-(\hat{\A}+\alpha I)(t-s)} dw(s)\right|^2_X \nonumber \\
       &\le C \int_{-\infty}^t \| e^{-(\hat{\A}  + \alpha I)(t-s)}\|^2_{R(K,X)} ds
        \label{equ:Eineq1}\\
       &\le C \int_0^\infty e^{-2\alpha s} \|e^{-s\hat{\A}}\|^2_{R(K,X)} ds.
       \label{equ:Eineq2}
\end{align}
Moreover, $\E|\ovz_\alpha(t)|_X^2$ tends to $0$ as $\alpha \rightarrow \infty$.
\end{proposition}
\begin{proof}
Stationarity of the process $\ovz_\alpha$ follows from equation \eqref{equ:zalpha}.
The equality \eqref{equ:integral} follows by finite-dimensional approximation, and
inequality \eqref{equ:Eineq1}--\eqref{equ:Eineq2} follows from \cite{Brz97}.

By \dela{Lemma~\ref{rad0}}Corollary \ref{cor-rad0}
\begin{equation}\label{equ:BBB}
\int_0^\infty\|e^{-s\hat{ \A}}\|^2_{R(K,X)} ds < \infty.
\end{equation}
Hence, using the Cauchy-Schwarz inequality, we conclude that the last integral is finite.

Finally, the last statement follows from \eqref{equ:Eineq2} by applying the Lebesgue
Dominated Convergence Theorem.
\end{proof}

By Proposition~\ref{pro:stationary}, $\ovz_\alpha(t)$, $t\in \R$, is a stationary
and ergodic $X$-valued process, by the Strong Law for Large Numbers
(see Da Prato and Zabczyk \cite{PraZab96} for a similar argument),
\begin{equation}\label{equ:Ezalpha}
 \lim_{t \rightarrow \infty} \frac{1}{t} \int_{-t}^0 |\ovz_\alpha(s)|^2_X ds
  = \E |\ovz_\alpha(0)|^2_X, \quad \mbox{ a.s.}
\end{equation}

Denote by $\Omega_\alpha(\xi,E)$ the set of those $\omega \in \Omega(\xi,E)$ for which
the equality \eqref{equ:Ezalpha} holds true. It follows from Corollary~\ref{cor:flow} that
this set is invariant with respect to the flow $\theta$, i.e. for all $\alpha \ge 0$ and
all $t \in \R$, $\vartheta_t(\Omega_\alpha(\xi,E)) \subset \Omega_\alpha(\xi,E)$. Therefore,
the same is true for a set
\[
 \hat{\Omega}(\xi,E)= \bigcap_{n=0}^\infty \Omega_n(\xi,E).
\]
It follows that as a model for a metric dynamical system we can take either the
quadruple $(\Omega(\xi,E),\calF,\bbP,\vartheta)$ or the quadruple
$(\hat{\Omega}(\xi,E),\hat{\calF},\hat{\bbP},\hat{\vartheta})$, where
$\hat{\calF}$,$\hat{\bbP}$, and $\hat{\vartheta}$ are respectively the natural
restrictions of $\calF,\bbP$ and $\vartheta$ to $\hat{\Omega}(\xi,E)$.
\section{Random dynamical systems generated by the stochastic NSEs on the sphere}
In this section we use standard constructions and basic theory of random dynamical systems as presented in a fundamental monograph \cite{arnold}. For a short summary of facts relevant for this section the reader may consult \cite{langa} or \cite{BrzLi06}.
\par
We fix a parameter $\alpha \ge 0$ that we will
vary in the following sections. We also fix $\xi \in (\delta,1/2)$ and put
$\Omega = \Omega(\xi,E)$.

We define a map $\varphi = \varphi_\alpha : \R_{+} \times \Omega \times H \rightarrow H$ by
\begin{equation}\label{def:varphi}
(t,\omega,\x) \mapsto \vv(t,\ovz_\alpha(\omega))(\x-\ovz_\alpha(\omega)(0)) + \ovz_\alpha(\omega)(t).
\end{equation}
In what follows, we put for simplicity $\ovz= \ovz_\alpha$.

Because $\ovz(\omega) \in C_{1/2}(\R,X)$, $\ovz(\omega)(0)$ is a well-defined element of $H$
and hence $\varphi$ is well defined. Furthermore, we have the main result of this section.
\begin{theorem}
$(\varphi,\vartheta)$ is a random dynamical system.
\end{theorem}
\begin{proof}
All properties of a random dynamical system, except the cocycle property,  follow
from Theorem~\ref{thm:limit}. Hence we only need to show that for any $\x\in H$,
\[
 \varphi(t+s,\omega)\x = \varphi(t,\vartheta_s \omega) \varphi(s,\omega)\x, \quad t,s \in \R^{+}.
\]
The proof follows from apply similar techniques used in
\cite[Theorem 6.15]{BrzLi06} to the stochastic Navier-Stokes equations on the sphere.
\end{proof}

Suppose that Assumption~\ref{ass:radon} is satisfied. If $u_s \in H$, $s\in \R$, $f\in V^\prime$
and $W_t$, $t \in \R$ is a two-sided Wiener process introduced after \eqref{equ:Wiener}
such that the Cameron-Martin (or Reproducing Kernel Hilbert) space of the Gaussian measure
$\calL(w_1)$ is equal to $K$. A process $\vu(t)$, $t\ge 0$, with trajectories in
$C([s,\infty);H)\cap L^2_{\loc}([s,\infty);V)\cap L^2_{\loc}([s,\infty);\lL^4(\S))$ is a solution
to problem \eqref{sNSEs} iff $\vu(s) = \vu_s$ and for any $\vphi \in V$, $t > s$,
\be\label{def:sol sNSE}
\begin{aligned}
(\vu(t),\vphi) &= (\vu(s),\vphi) - \nu \int_s^t (\A\vu(r),\vphi)dr - \int_s^t b(\vu(r),\vu(r),\vphi)dr \\
               & - \int_s^t (\CC\vu(r),\vphi) dr + \int_s^t (\vf,\vphi)dr
               + \int_s^t \langle \vphi, dW_r \rangle.
\end{aligned}
\ee

\begin{proposition}
In the framework as above, suppose that $\vu(t) = \ovz_\alpha(t) + \vv_\alpha(t)$, $t \ge s$,
where $\vv_\alpha$ is the unique solution to problem \eqref{equ:ode} with
initial data $\vu_0 - \ovz_\alpha(s)$ at time $s$. If the process $\vu(t)$, $t \ge s$, has
trajectories in $C([s,\infty);H)\cap L^2_{\loc}([s,\infty);V)\cap L^2_{\loc}([s,\infty);\lL^4(\S))$, then
it is a solution to problem \eqref{sNSEs}. Vice-versa, if a process $\vu(t)$, $t \ge s$,
with trajectories in   $C([s,\infty);H)\cap L^2_{\loc}([s,\infty);V)\cap L^2_{\loc}([s,\infty);\lL^4(\S))$
is a solution to problem \eqref{sNSEs}, then for any $\alpha \ge 0$, a process $\vv_\alpha(t)$, $t \ge s$,
defined by $\ovz_\alpha(t) = \vu(t) - \vv_\alpha(t)$, $t\ge s$, is a solution to \eqref{equ:ode} on $[s,\infty)$.
\end{proposition}

Our previous results yield the existence and the uniqueness of solutions to problem \eqref{sNSEs} as well
as its continuous dependence on the data (in particular on the initial value $\vu_0$ and the force $\vf$).
Moreover, if we define, for $\x \in H$, $\omega \in \Omega$, and $t \ge s$,
\be\label{def:uv}
\vu(t,s; \omega,\vu_0) := \varphi(t-s; \vartheta_s \omega)\vu_0 =
 \vv (t,s; \omega,\vu_0 -\ovz(s)) + \ovz(t),
\ee
then for each $s\in \R$ and each $\vu_0 \in H$, the process $\vu(t)$, $t \ge s$, is a solution to problem
\eqref{sNSEs}.

Before presenting the main results of this section, we discuss the weak continuity of the RDS generated by the SNSE on the sphere.
\begin{lemma}\label{lem:weakcont1}
If $T>0$ then the map
\[
 H \ni \x \mapsto \vv(\cdot,\x) \in L^2([0,T];V)
\]
is continuous in the weak topologies.
\end{lemma}
\begin{lemma}\label{lem:weakcont2}
Let $T>0$. Then the maps
\[
  H \ni \x \mapsto \vv(t,\x) \in H,\quad t\in[0,T],
\]
are uniformly continuous in the weak topologies. More precisely, if $\x_n \rightarrow \x$ weakly in $H$,
then for any $\phi \in H$, $(\vv(\cdot,\x_n),\phi) \rightarrow (\vv(\cdot,\x),\phi)$ uniformly
on $[0,T]$, as $n\rightarrow \infty$.
\end{lemma}

We have the Poincar\'{e} inequalities
\be\label{equ:Poincare}
 \begin{aligned}
   \|\vu\|_V^2 &\ge \lambda_1 \|\vu\|^2, \quad \mbox{ for all } \vu \in V, \\
   \|\A\vu\|^2 &\ge \lambda_1 \|\vu\|^2, \quad \mbox{ for all } \vu \in \calD(A).
 \end{aligned}
\ee

For any $\vu,\vv \in V$, we define a new scalar product $[\cdot,\cdot]:V \times V \rightarrow \R$ by the formula
$[\vu,\vv] = \nu(\vu,\vv)_V - \nu \frac{\lambda_1}{2}(\vu,\vv)$. Clearly,
$[\cdot,\cdot]$ is bilinear and symmetric. From \eqref{equ:Poincare}, we can prove
that $[\cdot,\cdot]$ define an inner product in $V$ with the norm
$[\cdot]=[\cdot,\cdot]^{1/2}$, which is equivalent to the norm $\|\cdot\|_V$.

\begin{lemma}\label{lem:NSEsol}
Suppose that $\vv$ is a solution to problem \eqref{equ:ode} on the time interval $[a,\infty)$ with
$\ovz \in L^4_{\loc}(\R^+, \lL^4(\S)) \cap L^2_{\loc}(\R^+,V^\prime)$ and $\alpha \ge 0$. Denote
$\vg(t) =\alpha \ovz(t) - \B(\ovz(t),\ovz(t))$, $t\ in [a,\infty)$. Then, for any $t \ge \tau \ge a$,
\be\label{equ:8_10}
\begin{aligned}
\|\vv(t)\|^2 &\le \|\vv(\tau)\|^2
e^{-\nu \lambda_1(t-\tau)+ \frac{3C^2}{\nu} \int_\tau^t \|\ovz(s)\|^2_{\lL^4}) ds} \\
 & \frac{3}{\nu} \int_\tau^t ( \|\vg(s)\|^2_{V^\prime} + \|\vf\|^2 )
 e^{-\nu \lambda_1(t-\tau)+ \frac{3C^2}{\nu} \int_s^t \|\ovz(\xi)\|^2_{\lL^4}) d\xi}
 ds
\end{aligned}
\ee
\be\label{equ:8_11}
\begin{aligned}
\|\vv(t)\|^2 &= \|\vv(\tau)\|^2 e^{-\nu \lambda_1(t-\tau)} + \\
&2 \int_\tau^t e^{-\nu \lambda_1(t-s)}(b(\vv(s),\ovz(s),\vv(t)) +
\langle\vg(s),v(s)\rangle + \langle\vf,\vv(s)\rangle
    - [\vv(s)]^2) ds
\end{aligned}
\ee
Here the scalar product $[\cdot,\cdot]: V \times V \rightarrow \R$ is defined by the formula
$[\vu,\vv] = \nu (\A \vu,\vv) - \nu \frac{\lambda_1}{2} (\vu,\vv)$.
\end{lemma}

\begin{proof}
By \cite[Lemma III.1.2]{Tem79}, we have
$\frac{1}{2}\partial_t \|\vv(t)\| = (\vv(t),\vv(t))$. Hence
\begin{equation}\label{eqn:812}
\begin{aligned}
\frac{1}{2}\frac{d}{dt} \|\vv\|^2 &=
\nu(\A\vv,\vv) - (\CC\vv,\vv) - (\B(\vv,\vv),\vv) - (B(\ovz,\vv),\vv) \\
& \qquad -(B(\vv,\ovz),\vv) + \langle \vg,\vv\rangle +
\langle\vf,\vv\rangle\\
& = \nu \|\vv\|^2_V - b(\vv,\ovz,\vv) +
\langle\vg,\vv \rangle + \langle \vf,\vv\rangle.
\end{aligned}
\end{equation}
From \eqref{b_estimate3} and invoking the Young inequality,
we have
\begin{align*}
|b(\vv,\ovz,\vv)| &\le C\|\vv\|_{\mathbb{L}^4} \|\vv\|_V \|\ovz\|_{\mathbb{L}^4} \\
        &\le \frac{\nu}{6}\|\vv\|^2_V + \frac{3C^2}{2\nu} \|\vv\|^2
        \|\ovz\|^2_{\mathbb{L}^4},
\end{align*}
and
\begin{align*}
|\langle \vg,\vv\rangle+\langle\vf,\vv\rangle| &\le \|\vg\|_{V^\prime} \|\vv\|_V   + \|\vf\|_{V^\prime}\|\vv\|_{V} \\
 &\le \frac{\nu}{3}     \|\vv\|^2 + \frac{3}{2\nu} \|g\|^2_{V^\prime} + \frac{3}{2\nu}
 \|\vf\|^2_{V^\prime}.
\end{align*}
Hence from \eqref{eqn:812} and \eqref{equ:Poincare}, we get
\begin{align*}
\frac{d}{dt} \|\vv(t)\|^2 &\le  -\nu\| \vv(t) \|^2 +
\frac{3C^2}{\nu}\|\ovz(t)\|^2_{\mathbb{L}^4} \|\vv(t)\|^2 +
  \frac{3}{\nu}\|\vg(t)\|^2_{V^\prime} + \frac{3}{\nu} \|\vf\|^2_{V^\prime} \\
  &\le\left(-\nu\lambda_1 + \frac{3C^2}{\nu}\|\ovz(t)\|^2_{\mathbb{L}^4}\right)
  \|\vv(t)\|^2
  + \frac{3}{\nu}\|\vg(t)\|^2_{V^\prime} + \frac{3}{\nu} \|\vf\|^2_{V^\prime}.
\end{align*}
Next, using the Gronwall Lemma, we arrive at \eqref{equ:8_10}.
By
adding and subtracting $\nu \frac{\lambda_1}{2}\|\vv(t)\|^2$ from \eqref{eqn:812}
we find that
\begin{align}
\frac{d}{dt}\|\vv(t)\|^2 &+ \nu \lambda_1 \|\vv(t)\|^2 + 2[\vv(t)]^2\\
    & = 2b(\vv(t),\ovz(t),\vv(t)) + 2 \langle \vg(t),\vv(t)\rangle + 2\langle\vf(t),\vv(t)\rangle.
\end{align}
Hence \eqref{equ:8_11} follows by the variation of constants formula.
\end{proof}
\begin{proposition}\label{prop:ACprop}
The RDS $\varphi$ is asymptotically compact provided for any bounded set $B \subset H$ there exists
a closed and bounded random set $K(\omega)$ absorbing $B$.
\end{proposition}

Let us recall that the RDS $\varphi$ is independent of the auxiliary parameter $\alpha \in \nN$. For
reasons that will become clear later, we choose $\alpha$ such that
$\E\|\ovz_\alpha(0)\|^2_{\lL^4} \le \frac{\nu^2 \lambda_1}{6C^2}$, where $\ovz_\alpha(t)$, $t\in \R$,
is the Ornstein-Uhlenbeck process from Section~\ref{OU}, $C>0$ is a certain universal constant,
$\lambda_1$ is the constant from \eqref{equ:Poincare} and $\nu>0$ is the viscosity.

\noindent
\begin{proof}
Suppose that $B \subset H$ is a bounded set, $(t_n)_{n=1}^\infty$ is an increasing sequence of positive
numbers such that $t_n \rightarrow \infty$ and $(\x_n)_n$ is a $B$-valued sequence. By our assumptions we
can find a closed bounded random set $K(\omega)$ in $H$ that absorbs $B$. We fix $\omega \in \Omega$.\\
\emph{Step I. Reduction. }Since $K(\omega)$ absorbs $B$, for $n\in \nN$ sufficiently large,\linebreak
$\varphi(t_n,\vartheta_{-t_n} \omega) B \subset K(\omega)$. Since $K(\omega)$ is closed and bounded,
and hence weakly compact, without loss of generality we may assume that
$\varphi(t_n,\vartheta_{-t_n}\omega) B \subset K(\omega)$ for all $n \in \nN$ and, for some $\vy_0 \in K(\omega)$,
\be\lb{def:y0}
  \varphi(t_n,\vartheta_{-t_n} \omega) \x_n \rightarrow \vy_0 \quad \mbox{ weakly in } H.
\ee
Since $\ovz(0) \in H$, we also have
\[
  \varphi(t_n,\vartheta_{-t_n}\omega)\x_n - \ovz(0) \rightarrow \vy_0 - \ovz(0) \quad \mbox{ weakly in } H.
\]
In particular,
\be\label{equ:liminf}
 \|\vy_0 - \ovz(0)\| \le \liminf_{n\rightarrow \infty} \| \varphi(t_n,\vartheta_{-t_n}\omega) \x_n - \ovz(0)\|.
\ee

We claim that it is enough to prove that for some subsequence $\{ n' \} \subset \nN$
\be\label{equ:limsup}
\|\vy_0 - \ovz(0)\| \ge \limsup_{n'\rightarrow \infty} \| \varphi(t_{n'},\vartheta_{-t_{n'}}\omega) \x_{n'} - \ovz(0)\|.
\ee
Indeed, since $H$ is a Hilbert space, \eqref{equ:liminf} in conjunction with \eqref{equ:limsup} imply that
\[
 \varphi(t_n,\vartheta_{-t_n}\omega)\x_n - \ovz(0) \rightarrow \vy_0 - \ovz(0) \quad \mbox{ strongly in } H
\]
which implies that
\[
\varphi(t_n,\vartheta_{-t_n} \omega) \x_n \rightarrow \vy_0 \quad \mbox{ strongly in } H.
\]
Therefore, in order to show that $\{\varphi(t_n,\vartheta_{-t_n}\omega)\x_n\}_n$ is relatively compact in $H$ we need
to prove that \eqref{equ:limsup} holds true.

\emph{Step II. Construction of a negative trajectory}, i.e. a sequence $(\vy_n)_{n=-\infty}^0$ such that
$\vy_n \in K(\vartheta_n \omega)$, $n \in \zZ^{-}$, and $\vy_k = \varphi(k-n,\vartheta_n \omega)\vy_n$, $n<k \le 0$.

Since $K(\vartheta_{-1}\omega)$ absorbs $B$, there exists a constant $N_1(\omega) \in \nN$, such that
\[
 \{ \varphi(-1+t_n,\vartheta_{1-t_n} \vartheta_{-1} \omega)\x_n : n \ge N_1(\omega) \} \subset K(\vartheta_{-1}\omega).
\]
Hence we can find a subsequence $\{n'\} \subset \nN$ and $\vy_{-1} \in K(\vartheta_{-1}\omega)$ such that
\be\lb{def:ym1}
 \varphi(-1+t_{n'}, \vartheta_{-t_{n'}}\omega)\x_{n'} \rightarrow \vy_{-1} \mbox{ weakly in } H.
\ee
We observe that the cocycle property, with $t=1$, $s=t_{n'}-1$, and $\omega$ being replaced by $\vartheta_{-t_{n'}}\omega$,
reads as follows:
\[
\varphi(t_{n'},\vartheta_{-t_{n'}}\omega) = \varphi(1,\vartheta_{-1}\omega) \varphi(-1+t_{n'},\vartheta_{t_{n'}}\omega).
\]
Hence, by Lemma~\ref{lem:weakcont2}, from \eqref{def:y0} and \eqref{def:ym1} we infer that
$\varphi(1,\vartheta_{-1}\omega) \vy_{-1} = \vy_0$. By induction, for each $k=1,2,\ldots,$ we can construct a subsequence
$\{n^{(k)} \} \subset \{ n^{(k-1)} \}$ and $\vy_{-k} \in K(\vartheta_{-k}\omega)$, such that
$\varphi(1,\vartheta_{-k}\omega)\vy_{-k} = \vy_{-k+1}$ and
\be\lb{def:ymk}
\varphi(-k+ t_{n^{(k)}}, \vartheta_{-t_{n^{(k)}}}\omega)\x_{n^{(k)}} \rightarrow \vy_{-k} \mbox{ weakly in } H,
 \mbox{ as } n^{(k)} \rightarrow \infty.
\ee

As above, the cocycle property with $t=k$, $s=t_{n^{(k)}}$ and $\omega$ being replaced by $\vartheta_{-t_{n^{(k)}}} \omega$
yields
\be\label{equ:varphi2}
\varphi( t_{n^{(k)}}, \vartheta_{- t_{n^{(k)}}  } \omega)
= \varphi(k,\vartheta_{-k}\omega) \varphi( t_{n^{(k)}} - k, \vartheta_{ - t_{n^{(k)}}  } \omega), \quad k \in \nN.
\ee
Hence, from \eqref{def:ymk} and by applying Lemma~\ref{lem:weakcont1}, we get
\be\label{equ:weaklim}
\begin{aligned}
 \vy_0 &= \mbox{w} - \lim_{ n^{(k)} \rightarrow \infty}
 \varphi( t_{n^{(k)}}, \vartheta_{ - t_{n^{(k)}}  } \omega) \x_{ n^{(k)}} \\
       &= \mbox{w} - \lim_{ n^{(k)} \rightarrow \infty}
       \varphi(k,\vartheta_{-k}\omega) \varphi( t_{n^{(k)}} - k, \vartheta_{ - t_{n^{(k)}}  } \omega) \x_{ n^{(k)}} \\
       &= \varphi(k,\vartheta_{-k}\omega)
         (  \mbox{w} - \lim_{ n^{(k)} \rightarrow \infty}
           \varphi( t_{n^{(k)}} - k, \vartheta_{ - t_{n^{(k)}}  } \omega) \x_{ n^{(k)}} ) \\
       &= \varphi(k,\vartheta_{-k}\omega) \vy_{-k},\\
\end{aligned}
\ee
where w-$\lim$ denotes the limit in the weak topology on $H$. The same proof yields a more general property:
\[
\varphi(j, \vartheta_{-k} \omega) \vy_{-k} = \vy_{-k+j} \mbox {  if } 0 \le j \le k.
\]

Before continuing with the proof, let us point out that \eqref{equ:weaklim} means precisely that
$\vy_0 = \vu(0,-k;\omega,\vy_{-k})$, where $\vu$ is defined in \eqref{def:uv}.\\
\emph{Step III. Proof of \eqref{equ:limsup}. }From now on, unless explicitly stated, we fix $k \in \nN$, and we will consider
problem \eqref{sNSEs} on the time interval $[-k,0]$. From \eqref{def:uv} and \eqref{equ:varphi2}, with $t=0$
and $s=-k$, we have
\be\label{equ:8_15}
\begin{aligned}
&\|\varphi( t_{n^{(k)}}, \vartheta_{- t_{n^{(k)}}  } \omega)\x_{ n^{(k)} } - \ovz(0)\|^2  \\
& \qquad = \| \varphi(k,\vartheta_{-k}\omega) \varphi( t_{n^{(k)}} - k, \vartheta_{ - t_{n^{(k)}}  }  \omega)\x_{ n^{(k)} }
       - \ovz(0)\|^2 \\
& \qquad = \| \vv(0,-k;\omega, \varphi( t_{n^{(k)}} - k, \vartheta_{ - t_{n^{(k)}}  }  \omega)\x_{ n^{(k)}} - \ovz(-k))\|^2.
\end{aligned}
\ee

Let $\vv$ be the solution to \eqref{equ:ode} on $[-k,\infty)$ with $\ovz=\ovz_\alpha(\cdot,\omega)$
and the initial condition at time $-k$:
$\vv(-k) = \varphi( t_{n^{(k)}} - k, \vartheta_{ - t_{n^{(k)}} } \omega)\x_{n^{(k)}} - \ovz(-k)$.
In other words,
\[
 \vv(s) =
 \vv\big(s,-k;\omega, \varphi( t_{n^{(k)}} - k, \vartheta_{ - t_{n^{(k)}} } \omega)\x_{n^{(k)}} - \ovz(-k) \big),
 \quad s \ge -k.
\]
From \eqref{equ:8_15} and \eqref{equ:8_11} with $t=0$ and $\tau=-k$ we infer that
\be\label{equ:8_16}
\begin{aligned}
&\|\varphi(t_{n^{(k)}},\vartheta_{ - t_{n^{(k)}} } \omega)\x_{n^{(k)}} - \ovz(0)\|^2 =
e^{-\nu \lambda_1 k}\|\varphi( t_{n^{(k)}} - k, \vartheta_{ - t_{n^{(k)}} } \omega)\x_{n^{(k)}} - \ovz(-k)\|^2 \\
&+ 2 \int_{-k}^0 e^{\nu \lambda_1 s} (b(\vv(s),\ovz(s),\vv(s)) + \langle \vg(s),\vv(s) \rangle
+ \langle \vf, \vv(s) \rangle - [\vv(s)]^2) ds.
\end{aligned}
\ee

It is enough to find a nonnegative function $h \in L^1(-\infty,0)$ such that
\be\label{equ:8_17}
\limsup_{ n^{(k)} \rightarrow \infty } \|\varphi(t_{n^{(k)}},\vartheta_{ - t_{n^{(k)}} } \omega)\x_{n^{(k)}} - \ovz(0)\|^2 \le \int_{-\infty}^{-k} h(s) ds + \|\vy_0 - \ovz(0)\|^2.
\ee
For, if we define the diagonal process $(m_j)_{j=1}^\infty$ by $m_j = j^{(j)}$, $j\in \nN$, then for each
$k\in \nN$, the sequence $(m_j)_{j=k}^\infty$ is a subsequence of the sequence $(n^{(k)})$ and hence by
\eqref{equ:8_17},
$
 \limsup_j\|\varphi(t_{m_j},\vartheta_{ - t_{m_j} } \omega)\x_{m_j} - \ovz(0)\|^2
 \le \int_{-\infty}^{-k} h(s) ds + \|\vy_0 - \ovz(0)\|^2.
$ Taking the $k\rightarrow \infty$ limit in the last inequality we infer that
\[
 \limsup_j\|\varphi(t_{m_j},\vartheta_{ - t_{m_j} } \omega)\x_{m_j} - \ovz(0)\|^2
  \le \|\vy_0 - \ovz(0)\|^2,
\]
which proves claim \eqref{equ:limsup}.\\
\emph{Step IV. Proof of \eqref{equ:8_17}. }We begin with estimating the first term on the RHS
of \eqref{equ:8_16}. If $-t_{n^{(k)}} < -k$, then by \eqref{def:uv} and \eqref{equ:8_10} we infer that
\be\label{equ:8_18}
\begin{aligned}
&\|\varphi(t_{n^{(k)}} - k, \vartheta_{-t_{n^{(k)}} } \omega) \x_{n^{(k)}} - \ovz(-k)\|^2 \\
&\qquad= \|\vv(-k,-t_{n^{(k)}};\vartheta_{-k}\omega, \x_{n^{(k)}} - \ovz(-t_{n^{(k)}})\|^2 e^{-\nu \lambda_1 k}  \\
&\le e^{-\nu \lambda_1 k} \big\{ \| \x_{n^{(k)}} - \ovz(-t_{n^{(k)}}) \|^2
                e^{ -\nu\lambda_1(t_{n^{(k)}} -k)+ \frac{3C^2}{\nu} \int_{-t_{n^{(k)}} }^{-k} \|\ovz(s)\|^2_{\lL^4} ds} \\
& \qquad + \frac{3}{\nu} \int_{-t_{n^{(k)}} }^{-k} [\|\vg(s)\|^2_{V^\prime} + \|\vf\|^2_{V^\prime} ]
                              e^{-\nu \lambda_1(-k-s) + \frac{3C^2}{\nu} \int_s^{-k} \|\ovz(\zeta)\|^2_{\lL^4}d\zeta}   \big\}  \\
&\le 2I_{ n^{(k)}} + 2 II_{n^{(k)}}+ \frac{3}{\nu} III_{n^{(k)}}+ \frac{3}{\nu} IV_{n^{(k)}},
\end{aligned}
\ee
where
\begin{align*}
I_{ n^{(k)}}  &=\|\x_{n^{(k)}}\|^2 e^{ -\nu \lambda_1 t_{n^{(k)}} +
             \frac{3C^2}{\nu} \int_{-t_{n^{(k)}} }^{-k} \|\ovz(s)\|^2_{\lL^4} ds  }  \\
II_{ n^{(k)}} &= \|\ovz( t_{n^{(k)}} )\|^2 e^{ -\nu \lambda_1 t_{n^{(k)}} +
             \frac{3C^2}{\nu} \int_{-t_{n^{(k)}} }^{-k} \|\ovz(s)\|^2_{\lL^4} ds  }  \\
III_{ n^{(k)}}&= \int_{-t_{n^{(k)}} }^{-k} \|\vg(s)\|^2_{V^\prime} e^{ -\nu\lambda_1 s + \frac{3C^2}{\nu}
                      \int_s^{-k}\|\ovz(\zeta)\|^2_{\lL^4}d\zeta  } \\
IV_{ n^{(k)}} &=\int_{-t_{n^{(k)}} }^{-k} \|\vf(s)\|^2_{V^\prime} e^{ -\nu\lambda_1 s + \frac{3C^2}{\nu}
                      \int_s^{-k}\|\ovz(\zeta)\|^2_{\lL^4}d\zeta  }
\end{align*}

We will find a nonnegative function $h \in L^1(-\infty,0)$ such that
\be\label{equ:8_19}
\limsup_{ n^{(k)} \rightarrow \infty} \| \varphi( t_{n^{(k)}} - k,\vartheta_{-t_{n^{(k)}}}\omega)\x_{n^{(k)}} - \ovz(-k)\|^2
  e^{-\nu \lambda_1 k} \le \int_{-\infty}^{-k} h(s) ds, \quad  k \in \nN.
\ee
This will be accomplished as soon as we prove the following four lemmas.

\begin{lemma}\label{lem:8_4}
$\limsup_{ n^{(k)} \rightarrow \infty } I_{n^{(k)}} = 0.$
\end{lemma}

\begin{lemma}\label{lem:8_5}
$\limsup_{ n^{(k)} \rightarrow \infty } II_{n^{(k)}} = 0.$
\end{lemma}

\begin{lemma}\label{lem:8_6}
$\int_{-\infty}^0 \|\vg(s)\|^2_{V^\prime} e^{ -\nu\lambda_1 s + \frac{3C^2}{\nu}
                      \int_s^{0}\|\ovz(\zeta)\|^2_{\lL^4}d\zeta  } < \infty.$
\end{lemma}

\begin{lemma}\label{lem:8_7}
$\int_{-\infty}^0  e^{ -\nu\lambda_1 s + \frac{3C^2}{\nu}
                      \int_s^{0}\|\ovz(\zeta)\|^2_{\lL^4}d\zeta  } < \infty.$
\end{lemma}

\begin{proof}[Proof of Lemma~\ref{lem:8_4}]
We recall that for $\alpha \in \nN$, $\ovz(t) = \ovz_\alpha(t)$, $t \in \R$, being the
Ornstein-Uhlenbeck process from Section \ref{OU}, one has
\[
 \E\|\ovz(0)\|^2_X = \E \|\ovz_\alpha(0)\|^2_X < \frac{\nu^2\lambda_1}{6 C^2}.
\]
Let us recall that the space $\hat{\Omega}(\xi,E)$ was constructed in such a way that
\[
 \lim_{ n^{(k)} \rightarrow \infty} \frac{1}{-k-(-t_{n^{(k)}})} \int_{t_{n^{(k)}} }^{-k}
  \| \ovz_\alpha(s)\|^2_X ds = \E\|\ovz(0)\|^2_X < \infty.
\]
Therefore, since the embedding $X \hookrightarrow \lL^4(\S)$ is a contraction, we have for $n^{(k)}$ sufficiently large,
\be\label{equ:8_20}
 \frac{3C^2}{\nu} \int_{t_{n^{(k)}} }^{-k}\| \ovz_\alpha(s)\|^2_{\lL^4} ds < \frac{\nu\lambda_1}{2}(t_{n^{(k)}} - k).
\ee
Since the set $B$ is bounded in $H$, there exists $\rho_1 >0$ such that for all $n^{(k)}$,
$\|\x_{n^{(k)}}\| \le \rho_1$. Hence
\be\label{equ:8_21}
 \limsup_{ n^{(k)} \rightarrow \infty} \|\x_{n^{(k)}}\|^2 e^{ -\nu \lambda_1 t_{n^{(k)}} +
              \frac{3C^2}{\nu} \int_{-t_{n^{(k)}} }^{-k} \|\ovz(s)\|^2_{\lL^4} ds  }
              \le \limsup_{ n^{(k)} \rightarrow \infty} \rho_1^2 e^{ -\frac{\nu\lambda_1}{2}(t_{n^{(k)}} - k)  } = 0.
\ee
\end{proof}

\begin{proof}[Proof of Lemma~\ref{lem:8_7}]
We denote by $p(s) = \nu \lambda_1 s + \frac{3C^2}{\nu} \int_s^0 \|\ovz(s)\|^2_{\lL^4}$. As in the proof of
Lemma~\ref{lem:8_4} we have, for $s\le s_0$, $p(s) < \frac{\nu \lambda_1}{2} s$. Hence $\int_{-\infty}^0 e^{p(s)}ds < \infty$,
as required.
\end{proof}

\begin{proof}[Proof of Lemma~\ref{lem:8_5}]
Because of \eqref{equ:zhatgrow}, we can find $\rho_2 \ge 0$ and $s_0 <0$, such that,
\be\label{equ:8_22}
\max\left( \frac{\|\ovz(s)\|} { |s|} , \frac{\|\ovz(s)\|_{V^\prime}}{|s|}, \frac{\|\ovz(s)\|_{\lL^4}}{|s|} \right) \le \rho_2,
\quad \mbox{ for } s \le s_0.
\ee
Hence by \eqref{equ:8_20} we infer that
\be\label{equ:8_23}
\begin{aligned}
&\limsup_{ n^{(k)} \rightarrow \infty} \|\ovz(-t_{n^{(k)}})\|^2
       e^{\int_{-t_{n^{(k)}} }^{-k} (-\nu \lambda_1 + \frac{3C^2}{\nu}\|\ovz(s)\|^2) ds  } \\
& \qquad \le \limsup_{ n^{(k)} \rightarrow \infty} \frac{\|\ovz(-t_{n^{(k)}})\|^2 }{|t_{n^{(k)}}|^2}
             \limsup_{ n^{(k)} \rightarrow \infty} |t_{n^{(k)}}|^2  e^{-\frac{\nu\lambda_1}{2}( t_{n^{(k)}}-k ) }  \le 0.
\end{aligned}
\ee
This concludes the proof of Lemma~\ref{lem:8_5}.
\end{proof}

\begin{proof}[Proof of Lemma~\ref{lem:8_6}]
Since
\[\|\vg(s)\|^2_{V^\prime} = \|\alpha \ovz(s) + 2\B(\ovz(s))\|^2_{V^\prime} \le 2 \alpha^2 \|\ovz(s)\|^2_{V^\prime} + 2C \|\ovz(s)\|^4_{\lL^4},\]
we only need to show that
\[
\int_{-\infty}^0 \|\ovz(s)\|^4_{\lL^4}e^{\nu\lambda_1 s + \frac{3C^2}{\nu}\int_s^0\|\ovz(\zeta)\|^2_{\lL^4}d\zeta} ds +
\int_{-\infty}^0 \|\ovz(s)\|^2_{V^\prime}e^{\nu\lambda_1 s + \frac{3C^2}{\nu}\int_s^0\|\ovz(\zeta)\|^2_{\lL^4}d\zeta} ds<\infty.
\]

It is enough to consider the case of $\|\ovz(s)\|^4_{\lL^4}$ since the proof will be similar for the remaining case.
Reasoning as in \eqref{equ:8_20}, we can find $t_0 \ge 0$ such that for $ t \ge t_0$,
\[
 \int_{-t}^{-t_0} \left( -\nu\lambda_1 + \frac{3C^2}{\nu} \|\ovz(\zeta)\|_{\lL^4}^2   \right) d\zeta
  \le -\frac{\nu\lambda_1}{2} (t-t_0).
\]
Taking into account the inequality \eqref{equ:8_22}, we have $\|\ovz(t)\| \le \rho_2(1+|t|)$, $t \in \R$.
Therefore, with
\[\rho_3 := \exp( \int_{-t_0}^0 (-\nu\lambda_1 + \frac{3C^2}{\nu}\|\ovz(\zeta)\|^2_{\lL^4})d\zeta,\]
we have
\[
\begin{aligned}
&\int_{-\infty}^{-t_0} \|\ovz(s)\|^4_{\lL^4}
   e^{\int_{s}^0 (\nu\lambda_1  + \frac{3C^2}{\nu}\|\ovz(\zeta)\|^2_{\lL^4})d\zeta} ds \\
& \qquad = \rho_3 \int_{-\infty}^{-t_0} \|\ovz(s)\|^4_{\lL^4}
   e^{\int_{s}^{-t_0} (\nu\lambda_1  + \frac{3C^2}{\nu} \|\ovz(\zeta)\|^2_{\lL^4})d\zeta} ds \\
& \qquad \le \rho_2^4 \rho_3 e^{\nu\lambda_1 t_0/2}  \int_{-\infty}^{t_0} |s|^4 e^{\nu\lambda_1 s/2} ds < \infty.
\end{aligned}
\]
By the continuity of all relevant functions, we can let $t_0 \rightarrow 0$ to get the result.
\end{proof}
This concludes the proof of \eqref{equ:8_19}, and it only remains to complete the proof of \eqref{equ:8_17}.
Let us denote by
\begin{align*}
\vv_\nk(s) &= \vv(s,-k;\omega, \varphi(t_\nk - k, \vartheta_{-t_\nk}\omega) \x_\nk - \ovz(-k)), \;\; s \in (-k,0),\\
\vv_k(s) &= \vv(s,-k;\omega, \vy_{-k} - \ovz(-k)), \;\; s \in (-k,0).
\end{align*}

From \eqref{def:ymk} and Lemma~\ref{lem:weakcont1} we infer that
\be\label{equ:8_24}
\vv_{\nk} \rightarrow \vv_k \mbox{ weakly in } L^2(-k,0;V).
\ee

Since $e^{\nu\lambda_1 \cdot} \vg, \; e^{\nu \lambda_1 \cdot} \vf \in L^2(-k,0;V^\prime)$, we get
\be\label{equ:8_25}
 \lim_{\nk \rightarrow \infty} \int_{-k}^0 e^{\nu\lambda_1 s} \langle \vg(s), \vv_{\nk}(s)\rangle ds
    = \int_{-k}^0 e^{\nu\lambda_1 s} \langle \vg(s),\vv_k(s)\rangle ds
\ee
and
\be\label{equ:8_26}
 \lim_{\nk \rightarrow \infty} \int_{-k}^0 e^{\nu\lambda_1 s} \langle \vf, \vv_{\nk}(s)\rangle ds
    = \int_{-k}^0 e^{\nu\lambda_1 s} \langle \vf,\vv_k(s)\rangle ds.
\ee

Using the same methods as those in the proof of Theorem~\ref{thm:existence}, we can find a subsequence of $\{\vv_\nk\}$, which, for the sake of notational simplicity simplicity, is still denoted as $\{\vv_\nk\}$, that satisfies
\be\label{equ:8_27}
 \vv_{\nk} \rightarrow \vv_k \mbox{ strongly in} L^2(-k,0;\lL^2_{\loc}(\S)).
\ee

Next, since $\ovz(t)$ is an $\lL^4$-valued process, so is $e^{\nu\lambda_1 t} \ovz(t)$. Thus by
Corollary~\ref{cor:bvmvmL4}, \eqref{equ:8_24} and \eqref{equ:8_27}, we infer that
\be\label{equ:8_28}
\begin{aligned}
&\lim_{\nk \rightarrow \infty} \int_{-k}^0 e^{\nu \lambda_1 s} b(\vv_{\nk}(s), \ovz(s),\vv_{\nk}(s)) ds \\
&\qquad =  \int_{-k}^0 e^{\nu \lambda_1 s} b(\vv_k(s),\ovz(s),\vv_k(s)) ds.
\end{aligned}
\ee

Moreover, since the norms $[\cdot]$ and $\|\cdot\|$ are equivalent on $V$, and for any $s \in (-k,0]$,
$e^{-\nu k \lambda_1} \le e^{\nu \lambda_1 s} \le 1$,
we invoke \eqref{equ:8_24} to obtain,
\[
\int_{-k}^0 e^{\nu \lambda_1 s} [\vv_k(s)]^2 ds \le
\liminf_{\nk\rightarrow \infty} \int_{-k}^0 e^{\nu\lambda_1s} [\vv_{\nk}(s)]^2 ds.
\]
Equivalently, ,
\be\label{equ:8_29}
\limsup_{\nk \rightarrow \infty}
\left(-\int_{-k}^0 e^{\nu\lambda_1 s} [\vv_{\nk}(s)]^2 ds \right)\le - \int_{-k}^0 e^{\nu\lambda_1 s} [\vv_k(s)]^2 ds.
\ee

From \eqref{equ:8_16}, \eqref{equ:8_19}, \eqref{equ:8_28} and \eqref{equ:8_29} we infer that
\be\label{equ:8_30}
\begin{aligned}
&\limsup_{\nk \rightarrow \infty} \|\varphi(t_{\nk},\vartheta_{-t_{\nk}}\omega)\x_{\nk} - \ovz(0)\|^2 \\
&\qquad \le \int_{-\infty}^{-k} h(s) ds + 2\int_{-k}^0 e^{\nu\lambda_1 s}
\big\{  b(\vv_k(s),\ovz(s),\vv_k(s)) \\
&\qquad \qquad \qquad+ \langle \vg(s),\vv_k(s) \rangle + \langle\vf,\vv_k(s)\rangle - [\vv_k(s)]^2 \big\} ds
\end{aligned}
\ee

On the other hand, from \eqref{equ:weaklim} and \eqref{equ:8_11}, we have
\be\label{equ:8_31}
\begin{aligned}
\|\vy_0 - \ovz(0)\|^2 &= \|\varphi(k,\vartheta_{-k}\omega)\vy_k - \ovz(0)\|^2 = \|\vv(0,-k;\omega,\vy_k-\ovz(-k))\|^2 \\
 &=\|\vy_k - \ovz(-k)\|^2 e^{-\nu\lambda_1 k} + 2 \int_{-k}^0 e^{\nu \lambda_1 s} \big\{ \langle \vg(s),\vv_k(s) \rangle\\
 &+ b(\vv_k(s),\ovz(s),\vv_k(s)) + \langle \vf,\vv_k(s) \rangle - [\vv_k(s)]^2 \big\} ds.
\end{aligned}
\ee

Hence, by combining \eqref{equ:8_30} with \eqref{equ:8_31}, we get
\[
\begin{aligned}
&\limsup_{\nk\rightarrow \infty} \|\varphi(t_{\nk}, \vartheta_{-t_{\nk}}\omega)\x_{\nk} - \ovz(0)\|^2 \\
&\quad\le \int_{-\infty}^{-k} h(s) ds + \|\vy_0 - \ovz(0)\|^2 - \|\vy_k - \ovz(-k)\|^2 e^{-\nu\lambda_1 k} \\
&\quad\le \int_{-\infty}^{-k} h(s) ds + \|\vy_0 - \ovz(0)\|^2,
\end{aligned}
\]
which proves \eqref{equ:8_17}, and Proposition~\ref{prop:ACprop} follows.
\end{proof}

\begin{theorem}\label{main}
Consider the metric dynamical system $\frakT = (\hat{\Omega}(\xi,E),\hat{\calF},\hat{\bbP},\hat{\vartheta})$,
and consider the RDS $\varphi$ over $\frakT$ generated by the stochastic Navier-Stokes equations on
the 2D-unit sphere \eqref{sNSEs} with additive
noise. Then the RDS $\varphi$ is asymptotically compact.
\end{theorem}

\begin{proof}
Let $B\subset H$ be a bounded set. In view of Proposition~\ref{prop:ACprop}, it is sufficient to prove that
there exists a closed bounded random set $K(\omega) \subset H$ which absorbs $B$. In fact, we will show below
that an even stronger property holds. Namely, that there exists a closed bounded random set $K(\omega) \subset H$
which absorbs every bounded deterministic set $B \subset H$.

Let $\omega \in \Omega$ be fixed. For a given $s\le 0$ and $\x\in H$, let $\vv$ be the solution of \eqref{equ:ode}
on $[s,\infty)$ with the initial condition $\vv(s) = \x - \ovz(s)$. Applying \eqref{equ:8_10} with $t=0$,
$\tau=s \le 0$, we get
\be\label{equ:8_32}
\begin{aligned}
\|\vv(0)\|^2 &\le 2\|\x\|^2 e^{\nu\lambda_1 s + \frac{3C^2}{\nu}\int_s^0 \|\ovz(r)\|^2_{\lL^4} dr}
                 + 2\|\ovz(s)\|^2 e^{\nu\lambda_1 s + \frac{3C^2}{\nu} \int_s^0 \|\ovz(r)\|^2_{\lL^4} dr} \\
& \quad + \frac{3}{\nu}\int_s^0 \{ \|\vg(t)\|^2_{V^\prime} + \|\vf\|^2_{V^\prime} \}
      e^{\nu\lambda_1 t + \frac{3C^2}{\nu} \int_s^0 \|\ovz(r)\|^2_{\lL^4} dr } dt
\end{aligned}
\ee
Set
\[
\begin{aligned}
r_1(\omega)^2 =
  &+ \sup_{s \le 0} \left\{ 2\|\ovz(s)\|^2 e^{\nu\lambda_1 s + \frac{3C^2}{\nu}\int_s^0 \|\ovz(r)\|^2_{\lL^4} dr} \right\}\\
  &+ \frac{3}{\nu} \int_{-\infty}^0 \int_s^0 \{ \|\vg(t)\|^2_{V^\prime} + \|\vf\|^2_{V^\prime} \}
        e^{\nu\lambda_1 t + \frac{3C^2}{\nu} \int_s^0 \|\ovz(r)\|^2_{\lL^4} dr } dt
\end{aligned}
\]
Similar to \eqref{equ:8_23} we can prove that
\[
\limsup_{t_0 \rightarrow -\infty} \|\ovz(t_0)\|^2 e^{\nu\lambda_1 t_0 + \frac{3C^2}{\nu}\int_{t_0}^0 \|\ovz(r)\|^2_{\lL^4}} = 0.
\]
Hence, by the continuity of the map $s\mapsto \ovz(s) \in H$,
\[
  \sup_{s\le 0} \|\ovz(s)\|^2 e^{\nu \lambda_1 s + \frac{3C^2}{\nu}\int_s^0 \|\ovz(r)\|^2_{\lL^4} dr} < \infty
\]
and in conjunction with Lemmas~\ref{lem:8_6} and \ref{lem:8_7}, we infer that
\be\label{equ:8_33}
 r_1^2(\omega) < \infty, \mbox{ for all } \omega \in \Omega.
\ee
On the other hand, given $\rho>0$, by \eqref{equ:8_21} we can find $t_\rho(\omega) \le 0$ such that, for all
$s \le t_\rho(\omega)$,
\[\rho^2 e^{\nu\lambda_1 s + \frac{3C^2}{\nu} \int_s^0 \|\ovz(r)\|^2_{\lL^4} dr} \le 1.\]
 Therefore,
from \eqref{equ:8_32}, if $\|\x\| \le \rho$ and $s \le t_\rho(\omega)$, then
\[\|\vv(0,s;\omega,\x-\ovz(s)\|^2 \le r_1^2(\omega).\]
Thus, we infer that
\[
 \|\vu(0,s;\omega,\x)\| \le \|\vv(0,s;\omega,\x - \ovz(s))\| + \|\ovz(0)\| \le r_2(\omega), \mbox{ for all } \omega \in \Omega,
\]
where $r_2(\omega) = r_1(\omega) + \|\ovz(0,\omega)\|$. From \eqref{equ:8_33} and our assumptions, we infer that for
all $\omega \in \Omega$, $r_2(\omega) < \infty$. Defining $K(\omega) := \{ \vu \in H: \|\vu\| \le r_2(\omega)\}$ concludes the
proof.
\end{proof}
\section{Appendix}
Theorem \ref{anal} below seems to be well known but we could not find a reference with the required statement. A sketch of the argument is presented below.
\par
We will consider differential operators acting in the space $L^p\la\mathbb S^2,T\mathbb S^2\ra$, where $p\in[1,\infty)$. Elements of $L^p\la\mathbb S^2,T\mathbb S^2\ra$ are $p$-integrable vector fields, that is
\[\|\theta\|_p=\la\int_{\mathbb S^2}|\theta(x)|^p\,dx\ra^{1/p},\quad\theta\in L^p\la\mathbb S^2,T\mathbb S^2\ra,\]
where $|\theta(x)|$ is the length of $\theta(x)$ in the tangent space $T_x$, and $dx$ denotes the Riemannian volume on $\mathbb S^2$. We will denote the L\'evy-Civita connection by $\nabla$. It can be shown that the operator $\nabla^\star\nabla$ defined on $C^\infty$ vector fields is essentially selfadjoint in $L^2\la\mathbb S^2,T\mathbb S^2\ra$. This fact is proved in \cite{strichartz} for the Laplace-Beltrami Laplacian and the proof for $\BL$ is an easy modification. The unique selfadjoint extension of $\nabla^\star\nabla$ is called the  \emph{Bochner Laplacian} or \emph{Rough Laplacian} and dented by $\BL$. The \emph{Hodge Laplacian} is given by $\HL=d^\star d+dd^\star$ and by the Weitzenbock formula
\[\HL=\BL+Ric,\]
where $Ric\in C\la\mathbb S^2,T\mathbb S^2\otimes T\mathbb S^2\ra$ is the Ricci tensor on the sphere. Let us note that on the two-dimensional sphere $Ric$ is equal to the metric tensor.
\par
Let  $V$ is a bounded and smooth potential and let $P$ denote the projection of $L^2\la\mathbb S^2,T\mathbb S^2\ra$ onto the kernel of divergence operator. Then on smooth vector fields $u$ defined on $\S$
\[P\HL u=\HL Pu.\]
Therefore
\[P\la\HL u+Vu\ra=\HL u+PVu.\]
For such an operator we can prove the following
\begin{theorem}\label{anal}
The operator $L=-\la \HL+PV\ra$ defined on $C^\infty$ vector fields extends to a generator of an analytic semigroup $\la T_t\ra$ in the space $L^p\la\mathbb S^2,T\mathbb S^2\ra$ for every $p\in(1,\infty)$.
\end{theorem}
\begin{proof} Let us fix $p \in (1,\infty)$.  It has been proved Theorem 2.4 in \cite{strichartz}, see also \cite{vesentini}, that the operator $\HL$ defined on $C^\infty$ vector fields is essentially self-adjoint in $L^2\la\mathbb S^2,T\mathbb S^2\ra$. We will still denote by $\HL$ the selfadjoint extension of this operator in $L^2\la\mathbb S^2,T\mathbb S^2\ra$. The strongly continuous semigroup $\la e^{-t\HL}\ra$ generated in $L^2\la\mathbb S^2,T\mathbb S^2\ra$ by $-\HL$ is dominated by a $C_0$ semigroup  $\la S^{\mathrm{LB}}(t)\ra$ generated by the Laplace-Beltrami operator $\LB$ acting in $L^2\la\mathbb S^2,\R\ra$. More precisely,
\begin{equation}\label{d1}
\left|e^{-t\HL}f(x)\right|\le S^{\mathrm{LB}}(t)|f|(x),\quad x\in\mathbb S^2,\,\, f\in L^2.
\end{equation}
This fact is essentially proved  in \cite{hess}, where the semigroup domination is proved for the Bochner Laplacian. Since in our case $Ric\ge 0$, the result easily follows. For a direct argument for $\HL$ in a more general situation see p. 106 of \cite{shigekawa}. Estimate \eqref{d1} implies immediately that $\la e^{-t\HL}\ra$ extends to $C_0$-semigroup of contractions on $L^p\la\mathbb S^2,T\mathbb S^2\ra$ for every $p\in[1,\infty)$. It is easy to see that the proof of Theorem 1 on p. 67 of \cite{stein} extends trivially to the semigroup $\la e^{-t\HL}\ra$ if \eqref{d1} holds and therefore $\la e^{-t\HL}\ra$ is analytic in  $L^p\la\mathbb S^2,T\mathbb S^2\ra$. Since $PV$ is a bounded operator on $L^p\la\mathbb S^2,T\mathbb S^2\ra$, the operator $-\HL-PV$ generates an analytic semigroup in every $L^p\la\mathbb S^2,T\mathbb S^2\ra$ by standard arguments, see for example \cite{Paz83}.
\end{proof}

\end{document}